\documentclass[12pt,reqno]{amsart}

\usepackage{amsfonts,amssymb}
\usepackage{hyperref}
\usepackage{epsfig, graphicx}
\usepackage{latexsym}
\usepackage{MnSymbol}
\usepackage{extarrows}
\usepackage{amsmath,amsthm}
\usepackage{mathrsfs}
\usepackage{MnSymbol}
\usepackage[all,cmtip]{xy}
\usepackage{mathtools}

\theoremstyle{plain}
\newtheorem{theorem}{Theorem}[section]
\newtheorem{lemma}[theorem]{Lemma}
\newtheorem{definition}[theorem]{Definition}
\newtheorem{proposition}[theorem]{Proposition}

\newtheorem{remark}[theorem]{Remark}

\numberwithin{equation}{section}

\newcommand{\ra}{\longrightarrow}
\newcommand{\dholoa}[1]{\mathcal{#1}}
\newcommand{\dholo}[1]{[\mathcal{U},{#1}_{\mathcal{U}}]}

\newcommand{\dholoc}[2]{{#1}_{#2}}
\newcommand{\twistC}{\mathcal{C}}
\newcommand{\cocyclecotangent}[3]{g_{#1#2}^{T_d^*{#3}}}

\newcommand{\cocyclebundle}[3]{g_{#1#2}^{#3}}
\newcommand{\overlinecocyclebundle}[3]{\overline{g}_{#1#2}^{#3}}
\newcommand{\cocycletangent}[3]{g_{#1#2}^{T_d{#3}}}

\newcommand{\xra}[1]{\xlongrightarrow{{#1}}}
\newcommand{\C}{\mathbb{C}}
\newcommand{\R}{\mathbb{R}}

\newcommand{\dholos}[1]{\mathfrak{#1}}
\newcommand{\bigslant}[2]{{\raisebox{.2em}{$#1$}\left/\raisebox{-.2em}{$#2$}\right.}}
\newcommand\restr[2]{{
  \left.\kern-\nulldelimiterspace 
  #1 
  \vphantom{|} 
  \right|_{#2} 
  }}

\newcommand{\dhstruct}[1]{\mathcal{O}_{#1}^{dh}}

\newcommand{\dhform}[1]{\Omega^1_{#1}}

\newcommand{\m}[1]{\mathfrak{m}_{#1}}

\hypersetup{colorlinks, citecolor=blue, filecolor=black, linkcolor=red}
\newcommand{\id}[1]{\ensuremath{\mathbf{1}_{#1}}}
\begin{document}
\baselineskip=15.5pt

\title{On $d$-holomorphic connections}
\author{Sanjay~Amrutiya}
\address{Department of Mathematics, IIT Gandhinagar,
 Near Village Palaj, Gandhinagar - 382355, India}
 \email{samrutiya@iitgn.ac.in}
\author{Ayush~Jaiswal}
\address{Department of Mathematics, IIT Gandhinagar,
 Near Village Palaj, Gandhinagar - 382355, India}
\email{ayush.jaiswal@iitgn.ac.in}
\thanks{This work is a part of the second author's Ph. D. thesis.}
\date{}

\begin{abstract}
\noindent We develop the theory of $d$-holomorphic connections on $d$-holomorphic 
vector bundles over a Klein surface by constructing the analogous Atiyah exact 
sequence for $d$-holomorphic bundles. We also give a criterion for the existence 
of $d$-holomorphic connection in $d$-holomorphic bundle over a Klein surface in 
the spirit of the Atiyah-Weil criterion for holomorphic connections.
\end{abstract}
\maketitle

\section{Introduction}\label{intro}
The theory of Klein surfaces is elaborated by Alling and Greenleaf in \cite{NLANG}. 
For $d$-holomorphic bundles over Klein surfaces, the Narasimhan-Seshadri-Donaldson 
correspondence is established in \cite{W}. In \cite{W1}, Wang has developed the 
twisted complex geometry, which generalizes the concept of $d$-holomorphic bundles
to higher dimension twisted complex manifolds. In \cite{IBMD}, Biswas and Dutta
have investigated various structures on non-orientable manifolds using the
orientation line bundle.

The purpose of this article is to develop the theory of $d$-holomorphic connections on $d$-holomorphic
bundle over a Klein surface following \cite{MFA, IBNR} by working within the framework of 
$d$-holomorphic structure.
We also give a criterion for the existence of $d$-holomorphic connection on $d$-holomorphic bundle
over a Klein surface. In Section \ref{sec-prelim}, we review some basic concepts on
Klein surfaces and $d$-holomorphic bundles. We discuss differential forms, the Dolbeault operator,
and study $d$-smooth connections and curvatures on $d$-complex vector bundles.
Using the orientation bundle, we have a generalized Hodge star operator, which plays an important role to 
get non-degenerate pairing (see \S \ref{subsec-criterion}). 
In Section \ref{sec-doperator}, we study the first order 
differential $d$-operator between $d$-holomorhic bundles and construct the symbol exact sequence. 
Using it, we give an analogous Atiyah exact sequence for a $d$-holomorphic bundle and define 
$d$-holomorphic connections as a splitting of the Atiyah exact sequence following \cite{MFA, IBNR}.
In Section \ref{sec-Jet}, we study Jet bundles and relate them with $d$-holomorphic connections. 
In Section \ref{sec-criterion}, we give a corresponding Atiyah-Weil criterion for the existence of 
$d$-holomorphic connection in a $d$-holomorphic bundle over a compact Klein surface.

\section{Klein surfaces}\label{sec-prelim}
A topological 2-manifold $X$ (with or without boundary) along with an atlas $\dholoa{U}=\{U_i,z_i\}_{i\in I}$ 
such that the transition maps $z_j\circ z_i^{-1}$ is either holomorphic or anti-holomorphic at each connected 
component of $U_i\cap U_j$, is called Klein surface or $d$-complex manifold of dimension 1 and the atlas is called the 
$d$-holomorphic atlas.

Two $d$-holomorphic atlases $\dholoa{U}_1=\{U_i,z_i\}_{i\in I}$ and $\dholoa{U}_2=\{U_j,z_j\}_{j\in J}$ are said to be 
equivalent if $\dholoa{U}_1\cup \dholoa{U}_2$ is again $d$-holomorphic atlas means $z_i\circ z_j^{-1}$ is either holomorphic 
or anti-holomorphic on each connected component of $U_i\cap U_j$ for some $i\in I$ and some $j\in J$. We will denote the 
maximal $d$-holomorphic atlas by $\dholos{X}$, which is called $d$-holomorphic structure on topological manifold $X$.

For a $d$-holomorphic atlas $\dholoa{U}=\{U_i,z_i\}_{i\in I}$ on $X$ and any open subset $U\subset X$, we have induced 
$d$-holomorphic atlas $\dholoa{U}\cap U=\{U_i\cap U,z_i|_{U_i\cap U}\}$ on $U$.

\begin{remark}\rm{
From now on w.l.g, we will assume that for a given $d$-holomorphic atlas $\dholoa{U}=\{(U_i,z_i)\}_{i\in I}$ each 
intersection $U_i\cap U_j$ ($i,j\in I$) is connected.
}
\end{remark}

For a given open subset $U$ of Klein surface with $d$-holomorphic atlas $\dholoa{U}=\{(U_i,z_i)\}_{i\in I}$, let  
$(\dholoa{U}\cap U,f_{\dholoa{U}\cap U})$ is a family of holomorphic functions 
$(\dholoc{f}{U_i\cap U}:U_i\cap U\ra \C)_{i\in I}$ with following compatibility condition
\begin{equation}\label{dholocompatibility}
  \dholoc{f}{U_j\cap U} = \left\{\begin{array}{lr}
        \dholoc{f}{U_i\cap U},  \hspace{10pt}\text{if } z_j\circ {z_i}^{-1} \text{ is holomorphic}\\
        \\
        \dholoc{\overline{f}}{U_i\cap U},  \hspace{10pt}\text{if } z_j\circ {z_i}^{-1} \text{ is anti-holomorphic}
        \end{array}\right.
\end{equation}
Note that the family depends on atlas $\dholoa{U}$; to make it independent, we will define an equivalence relation 
as follows.

Two families $(\dholoa{U}_1\cap U,f_{\dholoa{U}_1\cap U})$ and $(\dholoa{U}_2\cap U,f_{\dholoa{U}_2\cap U})$ are 
equivalent whenever atlases $\dholoa{U}_1$ and $\dholoa{U}_2$ are equivalent. Hence, we have an equivalence class 
containing the family $(\dholoa{U}_1\cap U,f_{\dholoa{U}\cap U})$, which is called the $d$-holomorphic function 
with respect to $d$-holomorphic structure induced by $\dholoa{U}_1$ on $X$ with domain $U$. This way, 
we get a sheaf of $d$-holomorphic functions on $X$, which we will denote by $\dhstruct{X}$. 
A Klein surface $(X,\dholos{X})$ can be viewed as an $\R$-ringed space with structure sheaf $\dhstruct{X}$.

\begin{remark}\rm{
 For a given $d$-holomorphic atlas $\dholoa{U}=\{U_i,z_i\}$, we will write a global $d$-holomorphic function as 
 $ f=\dholo{f}=\{\dholoc{f}{U_i}\}_{i\in I}$ and a $d$-holomorphic function with domain, some open subset $U\subset X$ 
 will be written as $f=[\dholoa{U}\cap U, f_{\dholoa{U}\cap U}]=\{f_{U_i\cap U}\}_{i \in I}$. 
 }
\end{remark}
\begin{remark}\rm{
For a given global $d$-holomorphic function $f$, it will be convenient to write $f|_{U_i}=\dholoc{f}{U_i}$, which is 
a holomorphic function or, $\dholoc{f}{U_i}\circ z_i^{-1}$ is holomorphic on $z_i(U_i)$. If $f\in \dhstruct{X}(U)$ for 
some open subset $U\subset X$, then $\{(U_i\cap U,z_i|_{U_i\cap U})\}_{i\in I}$ will be induced $d$-holomorphic 
atlas on $U$, and we will write $f|_{U_i\cap U}=\dholoc{f}{U_i\cap U}$.
}
\end{remark}

For a given Klein surface, one can construct its double cover $(\tilde{X}, \sigma)$, where $\tilde{X}$ is a Riemann 
surface with an anti-holomorphic involution $\sigma$, such the quotient space $\tilde{X}/\sigma$ with induced dianalytic 
structure is isomorphic to $X$ (see \cite{NLANG}).

From now on, if $a$ is some matrix, then its $(i,j)$th entry will be denoted by $(a)^i_j$ and $a=((a)^i_j)$.

\begin{definition}\cite{W}\label{s1d1}\rm{
A $d$-holomorphic vector bundle $E$ of rank $r$ on a Klein surface $(X,\dholos{X})$ is defined by its
transition maps $\cocyclebundle{i}{j}{E}:U_i\cap U_j\ra GL_n(r,\C)$ such that $\cocyclebundle{i}{j}{E}\circ z_i^{-1}$ is 
either holomorphic or antiholomorphic according to $z_j\circ z_i^{-1}$ is holomorphic or antiholomorphic, and satisfying 
the following co-cycle conditions
\[
 \cocyclebundle{k}{i}{E} = \left\{\begin{array}{ll}
 \cocyclebundle{k}{j}{E}\circ \cocyclebundle{j}{i}{E},  ~\hspace{10pt}\text{if } z_k\circ {z_j}^{-1} \text{ is holomorphic}\\
  \\
 \cocyclebundle{k}{j}{E}\circ \overlinecocyclebundle{j}{i}{E},  ~\hspace{10pt}\text{if } z_k\circ {z_j}^{-1} \text{ is anti-holomorphic}
 \end{array}\right.
 \]
along with
\[
  \cocyclebundle{i}{j}{E} = \left\{\begin{array}{lr}
        (\cocyclebundle{j}{i}{E})^{-1},  \hspace{10pt}\text{if } z_j\circ {z_i}^{-1} \text{ is holomorphic}\\
        \\
        (\overlinecocyclebundle{j}{i}{E})^{-1},  \hspace{10pt}\text{if } z_j\circ {z_i}^{-1} \text{ is anti-holomorphic}
        \end{array}\right.
  \]
  
The total space of $E$ is given by the quotient 
$\bigslant{(\displaystyle\coprod_{\{(U_i,z_i)\}_{i\in I}} U_i\times \C^r)}{\sim}$, where the equivalence relation $\sim$ 
is defined as follows:

Let $(x_i,\xi_i)\in U_i\times \C^r$ and $(x_j,\xi_j)\in U_j\times \C^r$ then $(x_i,\xi_i)\sim (x_j,\xi_j)$ if, 
\begin{equation}\label{section2equation1}
    x_i=x_j \text{ and }\xi_j = \left\{\begin{array}{lr}
       \cocyclebundle{j}{i}{E}(x_i)\xi_i,  \hspace{10pt}\text{if } z_j\circ {z_i}^{-1} \text{ is holomorphic}\\
        \\
        \cocyclebundle{j}{i}{E}(x_i)\overline{\xi}_i,  \hspace{10pt}\text{if } z_j\circ {z_i}^{-1} \text{ is anti-holomorphic}
        \end{array}\right. 
\end{equation}
}
\end{definition}  
A smooth $d$-complex bundle is defined by the transition maps $\cocyclebundle{i}{j}{E}$ which are 
required to be smooth with the above co-cycle conditions.  

For a given Klein surface $(X,\dholos{X})$ with dianalytic atlas $\dholoa{U}=\{(U_i,z_i)\}_{i\in I}$ and any open subset 
$U\subset X$, a section $s\in \Gamma(U,E)$ can be expressed as $[\dholoa{U}\cap U,s_{\dholoa{U}\cap U}]$, 
where $(\dholoa{U}\cap U,s_{\dholoa{U}\cap U})$ is a family $\{\dholoc{s}{U_i\cap U}:U_i\cap U\ra E\}_{i\in I}$ of 
holomorphic sections with compatibility conditions induced from gluing \eqref{section2equation1} as follows
\begin{equation}\label{compatibilitybundlesection}
  \dholoc{s}{U_j\cap U} = \left\{\begin{array}{lr}
        \dholoc{s}{U_i\cap U},  \hspace{10pt}\text{if } z_j\circ {z_i}^{-1} \text{ is holomorphic}\\
        \\
        \dholoc{\overline{s}}{U_i\cap U},  \hspace{10pt}\text{if } z_j\circ {z_i}^{-1} \text{ is anti-holomorphic}
        \end{array}\right.
\end{equation}
If $[\dholoc{s}{U_j\cap U}]$ (respectively, $[\dholoc{s}{U_i\cap U}]$) represents co-ordinate matrix of the section 
$\dholoc{s}{U_j\cap U}$ (respectively, $\dholoc{s}{U_i\cap U}$) with respect to frame $\{s_{jl}\}$(respectively, $\{s_{il}\}$) $(1\leq l\leq r)$, then we have
\begin{equation}\label{compatibilitybundlesection1}
  [\dholoc{s}{U_j\cap U}] = \left\{\begin{array}{lr}
        \cocyclebundle{j}{i}{E}[\dholoc{s}{U_i\cap U}],  \hspace{10pt}\text{if } z_j\circ {z_i}^{-1} \text{ is holomorphic}\\
        \\
        \cocyclebundle{j}{i}{E}\overline{[\dholoc{s}{U_i\cap U}]},  \hspace{10pt}\text{if } z_j\circ {z_i}^{-1} \text{ is anti-holomorphic}
        \end{array}\right.
\end{equation}
where $\cocyclebundle{j}{i}{E}:U_i\cap U_j\ra GL(r,\C)$ is co-cycle map associated to $d$-holomorphic bundle $E$.

\begin{remark}\rm{
Let $(X,\dholos{X})$ be a Klein surface with $d$-holomorphic atlas $\dholoa{U}=\{(U_i,z_i)\}$ and 
$E$ be $d$-holomorphic bundle of rank $r$ on $X$. Then, $E|_{U_i}$ will be trivial bundle and we 
have frame $\{s_{il}\}_{1\leq l\leq r}$, where $s_{il}\in \Gamma(U_i,E)(1\leq l\leq r)$. 
For two charts $(U_i,z_i)$ and $(U_j,z_j)$, let $\{s_{il}\}$ and $\{s_{jl}\}$, $(1\leq l\leq r)$ 
be frames for $E_{U_i}$ and $E_{U_j}$, respectively and $\cocyclebundle{j}{i}{E}:U_i\cap U_j\ra GL(r,\C)$ 
be co-cycle map, then we have
\begin{equation}\label{cocycle}
  \displaystyle\sum_{l=1}^r s_{il}(\cocyclebundle{i}{j}{E})^l_m = \left\{\begin{array}{lr}
        s_{jm},  \hspace{10pt}\text{if } z_j\circ {z_i}^{-1} \text{ is holomorphic}\\
        \\
        \overline{s}_{jm},  \hspace{10pt}\text{if } z_j\circ {z_i}^{-1} \text{ is anti-holomorphic}
        \end{array}\right.
\end{equation}
or, using Definition \ref{s1d1}
\begin{equation}\label{cocycle1}
s_{im} = \left\{\begin{array}{lr}
   \displaystyle\sum_{l=1}^r s_{jl}(\cocyclebundle{j}{i}{E})^l_m,  ~
   \hspace{10pt}\text{if } z_j\circ {z_i}^{-1} \text{ is holomorphic}\\
        \\
   \displaystyle\sum_{l=1}^r \overline{s}_{jl} (\overlinecocyclebundle{j}{i}{E})^l_m, ~
    \hspace{10pt}\text{if } z_j\circ {z_i}^{-1} \text{ is anti-holomorphic}
        \end{array}\right.
\end{equation}
}
\end{remark}

\begin{theorem}\cite[Theorem 1.4]{W}
Let $(X,\dholos{X})$ be a connected Klein surface. There is a 1-1 correspondence between $d$-holomorphic bundles  and a 
locally free sheaf of $\dhstruct{X}$-modules over $X$.
\end{theorem}

\subsection*{Differential forms}
Let $X$ be a Klein surface with $d$-holomorphic atlas $\dholoa{U}=\{(U_i,z_i)\}_{i\in I}$. 
The $d$-holomorphic tangent bundle to $X$ is 1-dimensional $d$-complex vector bundle with 1-cocycle map 
$\cocycletangent{j}{i}{E}:U_i\cap U_j\ra\C^*$ given as,
\[
    \cocycletangent{j}{i}{E} = \left\{\begin{array}{lr}
        (\frac{\partial z_j}{\partial z_i}),  \text{if } z_j\circ z_i^{-1} \text{ is holomorphic}\\
        \\
        (\frac{\partial z_j}{\partial \overline{z}_i}),  \text{if } z_j\circ z_i^{-1} \text{ is anti-holomorphic}
        \end{array}\right.
  \]
  which will be denoted by $T_dX$.
  
The $d$-holomorphic co-tangent bundle is $d$-complex vector bundle of dimension 1 with 1-cocycle map 
$\cocyclecotangent{j}{i}{E}:U_i\cap U_j\ra \C^*$ given as,
\[
    \cocyclecotangent{j}{i}{E} = \left\{\begin{array}{lr}
        (\frac{\partial z_i}{\partial z_j}), \hspace{10pt}
        \text{if } z_j\circ z_i^{-1} \text{ is holomorphic}\\
        \\
        (\frac{\partial \overline{z}_i}{\partial z_j}), \hspace{10pt} 
        \text{if } z_j\circ z_i^{-1} \text{ is anti-holomorphic}
        \end{array}\right.
  \]
  which will be denoted by $T_d^*X$.

It is important to note that a global section of $d$-holomorphic cotangent bundle on $X$ or global 
$d$-holomorphic 1-form is an equivalence class of a family of holomorphic 1-forms 
$(\dholoc{\omega}{U_i})_{i\in I}$ associated to $d$-holomorphic atlas $\dholoa{U}=\{(U_i,z_i)\}_{i\in I}$, 
which will be denoted by $\omega=[\dholoa{U},\omega_{\dholoa{U}}]=\{\dholoc{\omega}{U_i}\}_{i\in I}$ 
with the following compatibility 
conditions
\begin{equation}\label{compatibilitycotangentsection}
    \dholoc{\omega}{U_j} = \left\{\begin{array}{lr}
        \dholoc{\omega}{U_i}, \hspace{10pt}
         \text{if } z_j\circ z_i^{-1} \text{ is holomorphic}\\
        \\
        \dholoc{\overline{\omega}}{U_i}, \hspace{10pt}
         \text{if } z_j\circ z_i^{-1} \text{ is anti-holomorphic}
        \end{array}\right. 
\end{equation}
If $[\dholoc{\omega}{U_j}]$ (respectivey, $[\dholoc{\omega}{U_i}]$) represents a matrix associated to 
$\dholoc{\omega}{U_j}$ (respectively, $\dholoc{\omega}{U_i}$) with respect to frame $\{s_{jl}\}$ (respectively, 
$\{s_{il}\}$) $(1\leq l \leq r)$ and $\cocyclecotangent{j}{i}{X}:U_i\cap U_j\ra GL(r,\C)$ be associated
co-cycle map, then we have 

\begin{equation}\label{compatibilitycotangentsection1}
    [\dholoc{\omega}{U_j}] = \left\{\begin{array}{lr}
        \cocyclecotangent{j}{i}{E} [\dholoc{\omega}{U_i}], \hspace{10pt}
         \text{if } z_j\circ z_i^{-1} \text{ is holomorphic}\\
        \\
        \cocyclecotangent{j}{i}{E}\overline{[\dholoc{\omega}{U_i}]}, \hspace{10pt}
         \text{if } z_j\circ z_i^{-1} \text{ is anti-holomorphic}
        \end{array}\right. 
\end{equation}

Let $\Omega_d^1$ be the sheaf of $d$-holomorphic forms on $X$. Then, $\Omega_d^1(X) = \Gamma(X, T^*_dX)$.

A $d$-complex valued smooth $k$-form on $X$ is, by definition, a family of complex valued smooth forms 
$\{\omega_i\}_{i\in I}$ with respect to $d$-holomorphic atlas $\dholoa{U}=\{(U_i,z_i)\}_{i\in I}$ 
with the compatibility condition \eqref{compatibilitycotangentsection}. 
Let $\mathcal{A}^k_d(X)$ denote the space of 
$d$-complex valued smooth forms on $X$. Note that the space $\mathcal{A}^k(X)$ of smooth 
$k$-forms on $X$ is a subspaces $\mathcal{A}^k_d(X)$. Similarly, one can define the sheaf $\mathcal{A}_d^{(p, q)}$ 
of $d$-complex valued $(p, q)$-forms on $X$.

\subsection{Dolbeault operator}
\begin{lemma}
A $d$-complex valued smooth function $f=[\dholoa{U},f_{\dholoa{U}}]$ on Klein surface $X$ with 
$d$-holomorphic atlas $\dholoa{U}=\{(U_i,z_i)\}_{i\in I}$ is $d$-holomorphic if and only if 
$\overline{\partial}f\equiv 0$, where $\overline{\partial}:\mathcal{A}^0_d\ra \mathcal{A}^1_d$ is 
an operator such that 
$\overline{\partial}(\dholoc{f}{U_i})=\frac{\partial \dholoc{f}{U_i}}{\overline{z}_i}d\overline{z}_i$.
\end{lemma}
\begin{proof}
Let $X$ be a  Klein surface with $d$-holomorphic atlas $\dholoa{U}=\{(U_i,z_i)\}_{i\in I}$ and $f=\dholo{f}$ be 
$d$-smooth function. First, note that $\overline{\partial}(f)$ is $d$-complex valued smooth 1 form as the following 
compatibility conditions holds
\[
    \frac{\partial \dholoc{f}{U_j}}{\partial \overline{z}_j}= \left\{\begin{array}{lr}
        \frac{\partial \overline{z}_i}{\partial \overline{z}_j}\frac{\partial \dholoc{f}{U_{\beta}}}{\partial \overline{z}_i}
        =\frac{\partial \overline{z}_i}{\partial \overline{z}_j}\frac{\partial \dholoc{f}{U_i}}{\partial \overline{z}_i},  
        \hspace{10pt} \text{if } z_j\circ z_i^{-1} \text{ is holomorphic}\\
        \\
        \frac{\partial z_i}{\partial \overline{z}_j}\frac{\partial \dholoc{f}{U_j}}{\partial z_i}=
        \frac{\partial z_i}{\partial \overline{z}_j}\overline{(\frac{\partial \dholoc{f}{U_i}}{\partial \overline{z}_i})}, 
        \hspace{10pt} \text{if } z_j\circ z_i^{-1} \text{ is anti-holomorphic}
        \end{array}\right. 
  \]
 
Now, it is easy to note that a given $d$-complex valued smooth function $f=[\dholoa{U},f_{\dholoa{U}}]$ 
is $d$-holomorphic if and only if $\overline{\partial}(f)=0$.
\end{proof}

Since the operator $\overline{\partial}$ satisfies Leibnitz identity, we have the following definition.

\begin{definition}\rm{
The operator on Klein surface $X$, $\overline{\partial}:\mathcal{A}_d^0\ra \mathcal{A}_d^{(0,1)}$ taking a $d$-smooth 
function to $d$-complex valued smooth $(0,1)$ form is called the Cauchy-Riemann operator.
}
\end{definition}

\subsection*{Dolbeault operator on smooth $d$-complex bundle on Klein surface}
An $\R$-linear sheaf morphism
$$
D'':\mathcal{A}_d^0(E)\ra \mathcal{A}_d^{(0,1)}(E)
$$
satisfying the following identity:
$$
D''(fs)=(\overline{\partial}f)s+f\bigl(D''(s)\bigl)
$$
for $f\in\mathcal{A}^0_d$ and $s\in \mathcal{A}^0_d(E)$  is called Dolbeault operator for $d$-complex 
smooth bundle on Klein surface.

Let $E$ be a smooth $d$-complex bundle with a given $d$-holomorphic structure $\{(U_i,g_i)\}_{i\in I}$. 
Le $\mathcal{E}$ denote the associated $d$-holomorphic bundle, and let 
$[\dholoa{U},s_{\dholoa{U}}]\in \mathcal{A}^0_d(E)$, then we have
\[
    [\dholoc{s}{U_j}]= \left\{\begin{array}{lr}
        \cocyclebundle{j}{i}{\mathcal{E}}[\dholoc{s}{U_i}], \hspace{10pt}
         \text{if } z_j\circ z_i^{-1} \text{ is holomorphic}\\
        \\
        \cocyclebundle{j}{i}{\mathcal{E}}\overline{[\dholoc{s}{U_i}]}, \hspace{10pt}
         \text{if } z_j\circ z_i^{-1} \text{ is anti-holomorphic}
        \end{array}\right. 
  \]
where $\cocyclebundle{j}{i}{\mathcal{E}}:(U_i\cap U_j)\ra GL(r,\C)$ is co-cycle map, which is 
holomorphic (anti-holomorphic), whenever $z_j\circ z_i^{-1}$ is holomorphic (anti-holomorphic, respectively).

Now,
\begin{equation}\label{compatibilitycotangentbundle}
\frac{\partial [\dholoc{s}{U_j}]}{\partial \overline{z}_j}= \left\{\begin{array}{lr} 
\frac{\partial (\cocyclebundle{j}{i}{\mathcal{E}}[\dholoc{s}{U_i}])}{\partial \overline{z}_i}\frac{\partial \overline{z}_i}{\partial \overline{z}_j}
=\frac{\partial (\cocyclebundle{j}{i}{\mathcal{E}}[\dholoc{s}{U_i}])}{\partial \overline{z}_i}\frac{\partial \overline{z}_i}{\partial \overline{z}_j}
=\cocyclebundle{j}{i}{\mathcal{E}}\frac{\partial [\dholoc{s}{U_i}]}{\partial \overline{z}_i}\frac{\partial \overline{z}_i}{\partial \overline{z}_j}, 
\hspace{10pt} \text{if } z_j\circ z_i^{-1} \text{ is holomorphic}\\
        \\
\frac{\partial (\cocyclebundle{j}{i}{\mathcal{E}}\overline{[\dholoc{s}{U_i}]})}{\partial z_i}\frac{\partial z_i}{\partial \overline{z}_j}
=\frac{\partial (\cocyclebundle{j}{i}{\mathcal{E}}\overline{[\dholoc{s}{U_i}]})}{\partial z_i}\frac{\partial z_i}{\partial \overline{z}_j}
=\cocyclebundle{j}{i}{\mathcal{E}}\overline{(\frac{\partial [\dholoc{s}{U_i}]}{\partial \overline{z}_i})}\frac{\partial z_i}{\partial \overline{z}_j},  
\hspace{10pt} \text{if } z_j\circ z_i^{-1} \text{ is anti-holomorphic}
        \end{array}\right. 
\end{equation}
Since \eqref{compatibilitycotangentsection1} is being satisfied, for a given $d$-smooth section 
  $s=[\dholoa{U},s_{\dholoa{U}}]\in \mathcal{A}^0_d(E)$, we have an $E$ valued $d$-smooth (0,1)-form $\overline{\partial}_E(s)$ 
  such that, $\overline{\partial}_E(s)|_{U_i}=\overline{\partial}(\dholoc{s}{U_i})$ with the compatibility conditions \eqref{compatibilitycotangentbundle}

The Leibnitz type identity also holds because it holds for $\overline{\partial}$, hence we have a Cauchy-Riemann operator 
for $E$, $\overline{\partial}_E$ such that $\overline{\partial}_E|_{U_i}=\overline{\partial}$.

\subsection*{\v{C}ech description of Dolbeault isomorphism:}
Let $E$ be $d$-holomorphic vector bundle on a Klein surface $X$ and $\overline{\partial}_E$ be the Cauchy-Riemann 
operator for $E$. 
Then, we have
$$
\overline{\partial}_E(\sum_{k=1}^l s_k \dholoc{f}{U_j k})=
\sum_{k=1}^l s_k \otimes \frac{\partial \dholoc{f}{U_j k}}{\partial \overline{z}_j}d\overline{z}_j
$$
on chart $(U_j,z_j)$.

Let $Z_d^{(p,q)}(E)\subset \mathcal{A}_d^{(p,q)}(E)$ be the sheaf of $\overline{\partial}_E$ closed 
$d$-complex-valued smooth $(p,q)$ forms on $X$ with values in $E$. Then, we have the following short exact sequence
$$
0\ra \dhform{d}(E)\ra \mathcal{A}_d^{(1,0)}(E)\xrightarrow{\overline{\partial}_E} Z_d^{(1,1)}(E)\ra 0
$$
The above exact sequence induces a long exact sequence
$$
\dots\ra \mathcal{A}_d^{(1,0)}(E)(X)\xlongrightarrow{\overline{\partial}_E} Z_d^{(1,1)}(E)(X)
\xlongrightarrow{\delta} H^1\bigl(X,\dhform{d}(E)\bigl)\ra H^1\bigl(X,\mathcal{A}_d^{(1,0)}(E)\bigl)\ra\dots
$$
Since $\mathcal{A}_d^{(p,q)}(E)$ is fine sheaf, we can conclude that 
\begin{equation}\label{s4dolbeaultiso}
H^{(1,1)}(X,E)\cong H^1\bigl(X,\dhform{d}(E)\bigl).
\end{equation}

The above sheaf isomorphism can also be described using \v{C}ech cohomology as follows:
Let $\dholoa{U}=\{(U_i,z_i)\}_{i\in I}$ be $d$-holomorphic atlas on $X$. Let $a\in H^{(1,1)}(X,E)$ be a Dolbeault 
cohomology class, and $[\dholoa{U},\omega_{\dholoa{U}}]\in Z_d^{(1,1)}(E)(X)$ be a $\overline{\partial}_E$ 
closed-form representing $a$.

Since $E|_{U_i}$ is holomorphically trivial, by the Dolbeault-Grothendieck lemma, there exists smooth form 
$\tau_i\in \mathcal{A}_d^{(1,0)}(E)(U_i)$ for each $i\in I$ such that $\overline{\partial}_E(\tau_i)=\omega_i$. Let
\[
    u_{ij} = \left\{\begin{array}{lr}
        \tau_j|_{U_i\cap U_j}-\tau_i|_{U_{i}\cap U_j},  \text{if } z_j\circ z_i^{-1} \text{ is holomorphic}\\
        \\
        \tau_j|_{U_i\cap U_j}-\overline{\tau}_i|_{U_{i}\cap U_j},  \text{if } z_j\circ z_i^{-1} \text{ is anti-holomorphic}
        \end{array}\right.
  \]
then, $\overline{\partial}_E(u_{ij})=0$. On chart $(U_i\cap U_j\cap U_k, z_k)$, we have
\[
    u_{ik} = \left\{\begin{array}{lr}
        u_{jk}+u_{ij},  \text{if } z_k\circ z_j^{-1} \text{ is holomorphic}\\
        \\
        u_{jk}+\overline{u}_{ij},  \text{if } z_k\circ z_j^{-1} \text{ is anti-holomorphic}
        \end{array}\right.
  \]
Hence, the family $\{u_{ij}\}_{i,j\in I}$ satisfies the co-cycle condition and the associated cohomology class will be equal to $\delta(a)$.

\subsection{Connection and curvature}
Let $E$ be a smooth $d$-complex vector bundle on a Klein surface $(X,\dholos{X})$.

A $d$-smooth connection is an $\R$-linear sheaf morphism $D:\mathcal{A}^0_d(E)\ra \mathcal{A}^{1}_d(E)$ such that the following identity holds
$$
D(fs)=d(f)s+fD(s)
$$
where $f\in \mathcal{A}^0_d$ and $s\in \mathcal{A}^0_d(E)$. 

Let $s_{i}=\{s_{i 1},s_{i 2},\dots,s_{i r}\}$ be a local frame for $E$ over chart $(U_{i},z_{i})$ i.e.
\begin{enumerate}
\item $s_{i l}\in \mathcal{A}^0_d(E)(U_i)$ for all $l\in \{1,2,\dots,r\}$ (smooth sections)
\item $\{s_{i 1}(x),s_{i 2}(x),\dots,s_{i r}(x)\}$ is basis of $E(x)=E_x/\mathfrak{m}_xE_x$ for all $x\in U_i$.
\end{enumerate}

Then, for a given $d$-smooth connection $D$ on $E$ we can write,
$$
D(s_{i l})=\sum s_{i m}(\omega_{U_i})_m^l ~\text{ and }~
D(\overline{s}_{i l})=\sum \overline{s}_{i m}(\dholoc{\overline{\omega}}{U_i})_m^l
$$
where $(\dholoc{\omega}{U_i})_m^l\in \mathcal{A}^1_d(\dholoc{U}{i})$.

We call the matrix-valued $d$-smooth 1-form $\dholoc{\omega}{U_i}=(\dholoc{\omega}{U_i})_j^i$, 
the connection form of $D$ with respect to frame $s_{i}$ for $E_{\dholoc{U}{i}}$.

Now, we will see how connection form depends on the frame. Let $s_{\beta}=\{s_{\beta 1}, s_{\beta 2},\dots,s_{\beta r}\}$ be local frame over some another chart $(U_j,z_j)$, then there is a co-cycle map $\cocyclebundle{j}{i}{E}:U_{i}\cap U_{j}\ra GL(r,\C)$ such that the condition \eqref{cocycle1} holds.

Let $\dholoc{\omega}{U_i}$ and $\dholoc{\omega}{U_j}$ be connection forms with respect to frames $s_{i}=\{s_{i 1},s_{i 2},\dots, s_{i r}\}$ and $s_{j}=\{s_{j 1},s_{j 2},dots,s_{j r}\}$ for $E_{U_i}$ and $E_{U_j}$, respectivey. Then we have the compatibility codition  as follows
\[
    \dholoc{\omega}{U_j} = \left\{\begin{array}{lr}
        (\cocyclebundle{i}{j}{E})^{-1}.\dholoc{\omega}{U_i}.\cocyclebundle{i}{j}{E}+(\cocyclebundle{i}{j}{E})^{-1}. d(\cocyclebundle{i}{j}{E}),\hspace{10pt}  \text{if } z_{j}\circ z_{i}^{-1} \text{ is holomorphic}\\
        \\
        (\overlinecocyclebundle{i}{j}{E})^{-1}.\dholoc{\overline{\omega}}{U_i}.\overlinecocyclebundle{i}{j}{E}+(\overlinecocyclebundle{i}{j}{E})^{-1}. d(\overlinecocyclebundle{i}{j}{E}), \hspace{10pt} \text{if } z_{j}\circ z_{i}^{-1} \text{ is anti-holomorphic}
        \end{array}\right. 
\]
which can be described following the same line of arguements as in \eqref{connectioncompatibility}.

We can extend a connection $D:\mathcal{A}^0_d(E)\ra \mathcal{A}^1_d(E)$ to a $\R$-linear map $D:\mathcal{A}^p_d(E)\ra \mathcal{A}^{p+1}_d(E)$ for $(p>0)$ by setting,
$$
D(\sigma.\phi)=(D\sigma)\wedge\phi+\sigma.d(\phi)
$$
for $\sigma\in \mathcal{A}^0_d(E)$ and $\phi\in \mathcal{A}^p_d$.

Using the above extension, we can define the curvature $R$ associated to connection $D$ to be a $d$-complex valued 2-form taking values in $\mathcal{E}nd_{\dhstruct{X}}(E)$ as follows
$$
R=D\circ D:\mathcal{A}^0_d(E)\ra \mathcal{A}^2_d(E)
$$
On some chart $(U_i,z_i)$, if connection form of $D$ is $\dholoc{\omega}{U_i}$ with respect to frame 
$s_{i}=\{s_{i 1},s_{i 2},\dots,s_{i r}\}$ then curvature form $\dholoc{\Omega}{U_i}$ will be
\begin{equation}\label{cfc}
\dholoc{\Omega}{U_i}=d(\dholoc{\omega}{U_i})+\dholoc{\omega}{U_i}\wedge\dholoc{\omega}{U_i}
\end{equation}
which can be described using the same line of arguments as in \cite[1.12]{SK} 

If $\dholoc{\omega}{U_j}$ is connection form with respect to another frame $s_j$ for $E_{U_j}$
on chart $(U_j,z_j)$ and curvature form is $\dholoc{\Omega}{U_j}$, then the compatibility 
condition between $\dholoc{\Omega}{U_i}$ and $\dholoc{\Omega}{U_j}$ is given by
\begin{equation}\label{curc}
    \dholoc{\Omega}{U_j} = \left\{\begin{array}{lr}
        (\cocyclebundle{i}{j}{E})^{-1}.\dholoc{\Omega}{U_i}.\cocyclebundle{i}{j}{E},\hspace{10pt} 
         \text{if } z_j\circ z_i^{-1} \text{ is holomorphic}\\
        \\
        (\overlinecocyclebundle{i}{j}{E})^{-1}.\dholoc{\overline{\Omega}}{U_i}.\overlinecocyclebundle{i}{j}{E}, 
        \hspace{10pt} \text{if } z_j\circ z_i^{-1} \text{ is anti-holomorphic}
        \end{array}\right. 
\end{equation}
which is described as follows:
If $z_j\circ z_i^{-1}$ is anti-holomorphic, then
\[
\begin{aligned}
D^2(s_{jm})=D^2\bigl(\displaystyle\sum_l \overline{s}_{il}(\overlinecocyclebundle{i}{j}{E})^l_m\bigl)&=D\Bigl(\sum_lD(\overline{s}_{il})(\overlinecocyclebundle{i}{j}{E})^l_m+\sum_l\overline{s}_{il}d\bigl((\overlinecocyclebundle{i}{j}{E})^l_m\bigl)\Bigl)\\
&=\sum_lD^2(\overline{s}_{il})(\overlinecocyclebundle{i}{j}{E})^l_m\\
&=\sum_l(\sum_n\overline{s}_{in}(\dholoc{\overline{\Omega}}{U_i})^n_l)(\overlinecocyclebundle{i}{j}{E})^l_m\\
&=\sum_l(\sum_n(\sum_{n'} s_{jn'}(\cocyclebundle{i}{j}{E})^{n'}_n)(\dholoc{\overline{\Omega}}{U_i})^n_l)(\overlinecocyclebundle{i}{j}{E})^l_m\\
&=\sum_m s_{jm}(\cocyclebundle{j}{i}{E} \dholoc{\overline{\Omega}}{U_i}\overlinecocyclebundle{i}{j}{E})^m_l
\end{aligned}
\]
Similarly, when $z_j\circ z_i^{-1}$ is holomorphic, we have
$$
D^2(s_{jm})=\sum_m s_{jm}(\cocyclebundle{j}{i}{E}\dholoc{\Omega}{U_i}\cocyclebundle{i}{j}{E})^m_l.
$$
Hence, we have the compatibility condition \eqref{curc}.

\subsection*{Compatibility with $d$-holomorphic structure}
Let $D$ be a $d$-smooth connection on a smooth $d$-complex vector bundle $E$. 
Note that we have 
$$
\mathcal{A}^1_d=\mathcal{A}_d^{(1,0)}+\mathcal{A}_d^{(0,1)}
$$
and a connection $D$ decomposes as follows
$$
D=D^{(1,0)}+D^{(0,1)}
$$
where $D^{(1,0)}:\mathcal{A}^0_d(E)\ra \mathcal{A}_d^{(1,0)}(E)$ and $D^{(0,1)}:\mathcal{A}^0_d(E)\ra \mathcal{A}_d^{(0,1)}(E)$ are $\R$-linear sheaf morphisms satifying Leibnitz type identity.

A $d$-smooth connection on smooth $d$-complex bundle $E$ is said to be compatible with the $d$-holomorphic structure if $D^{(0,1)}\equiv \overline{\partial}_E$

Let $D$ be a $d$-smooth connection then its curvature $R$ will have decomposition as follows
$$
R=R^{(2,0)}+R^{(1,1)}+R^{(0,2)}
$$
where $R^{(p,q)}\in \mathcal{A}_d^{(p,q)}\bigl(\mathcal{E}nd_{\dhstruct{X}}(E)\bigl)$.

Let $(U_i,z_i)$, $(U_j,z_j)$ be two charts on $X$ and  $s_i=\{s_{il}\}$, $s_j=\{s_{jl}\}$ $(1\leq l\leq r)$ 
be frame for $E_{U_i}$ and $E_{U_j}$, respectively. Let $\cocyclebundle{j}{i}{E}:U_i\cap U_j\ra GL(r,\C)$ 
be co-cycle map with respect to frames $s_i$ and $s_j$. Let $\dholoc{\omega}{U_i}$ and $\dholoc{\Omega}{U_i}$ 
($\dholoc{\omega}{U_j}$ and $\dholoc{\Omega}{U_j}$) be connection and curvature forms with respect to frame 
$s_i$($s_j$, respectively). 

Note that $(\dholoc{\omega}{U_i})^l_m\in \mathcal{A}_d^{(1,0)}(U_i)$ and $R=D\circ D$ so, we have 
$R^{(0,2)}\circ R^{(0,2)}=\overline{\partial}_E\circ\overline{\partial}_E\equiv 0$. Therefore,
\begin{equation}\label{s4eq1}
\begin{aligned}
&\dholoc{\Omega}{U_i}^{(2,0)}=\partial \dholoc{\omega}{U_i}+\dholoc{\omega}{U_i}\wedge \dholoc{\omega}{U_i}\\
&\dholoc{\Omega}{U_i}^{(1,1)}=\overline{\partial}\dholoc{\omega}{U_i}\text{ and }\\
&\dholoc{\Omega}{U_i}^{(0,2)}=0
\end{aligned}
\end{equation}
which implies, $\overline{\partial}\dholoc{\Omega}{U_i}^{(1,1)}=0$, or $\overline{\partial}_ER^{(1,1)}=0$.

Thus, every $d$-smooth connection compatible with $d$-holomorphic structure defines a Dolbeault cohomology class $[R^{(1,1)}]\in H^{(1,1)}\bigl(X,\mathcal{E}nd_{\dhstruct{X}}(E)\bigl)$.

\subsection*{Hodge star operator}
Let $(X,\dholos{X})$ be a Klein surface and $\dholoa{U}=\{(U_i,z_i)\}_{i\in I}$ be dianalytic 
atlas on $X$. We can define a line bundle $L$ with same trivializing cover as dianalytic atlas has, 
along with the co-cycle map $s_{ij}:U_i\cap U_j\ra \{\pm{1}\}$ given as follows, 
\[
    s_{ij} = \left\{\begin{array}{lr}
        1,\hspace{10pt}  \text{if } z_j\circ z_i^{-1} \text{ is holomorphic}\\
        \\
        -1, \hspace{10pt} \text{if } z_j\circ z_i^{-1} \text{ is anti-holomorphic}
        \end{array}\right. 
  \]
This line bundle $L$ is known as the \emph{orientation bundle} (see \cite{W, IBMD}). 

A smooth \emph{$d$-Hermitian} metric $h$ on Klein surface is given by a family of $C^{\infty}$ tensors 
$\{\dholoc{h}{U_i}dz_i\wedge d\overline{z}_i\}_{i\in I}$ with the following compatibility condition
\begin{equation}\label{hermicomp}
    \dholoc{h}{U_j} = \left\{\begin{array}{lr}
        \dholoc{h}{U_i}\frac{\partial z_i}{\partial z_j}
        \frac{\partial \overline{z}_i}{\partial \overline{z}_j},\hspace{10pt}  
        \text{if } z_j\circ z_i^{-1} \text{ is holomorphic}\\
        \\
        \dholoc{\overline{h}}{U_i}\frac{\partial \overline{z}_i}{\partial z_j}
        \frac{\partial z_i}{\partial \overline{z}_j}, \hspace{10pt} 
        \text{if } z_j\circ z_i^{-1} \text{ is anti-holomorphic}
        \end{array}\right. 
\end{equation}
where $\dholoc{h}{U_i}:U_i\ra \C$ is smooth function along with 
$\dholoc{\overline{h}}{U_i}=\dholoc{h}{U_i}>0$. 

For a given $d$-Hermitian metric on Klein surface $(X,\dholos{X})$, say $h$, we can associate a 
real $(1,1)$ form $\Theta$ such that $\dholoc{\Theta}{U_i}=
\frac{i}{2}\dholoc{h}{U_i}dz_i\wedge d\overline{z}_i$ with the following compatibility condition
\begin{equation}\label{funda_form}
    \dholoc{\Theta}{U_j} = \left\{\begin{array}{lr}
        \dholoc{\Theta}{U_i},\hspace{10pt}  
        \text{if } z_j\circ z_i^{-1} \text{ is holomorphic}\\
        \\
        -\dholoc{\Theta}{U_i}, \hspace{10pt} 
        \text{if } z_j\circ z_i^{-1} \text{ is anti-holomorphic}
        \end{array}\right. 
\end{equation}
This real $(1,1)$ form is known as fundamental form associated to $d$-Hermitian metric $h$.
Let $(U_i,z_i)$ be some chart on $X$, where $z_i=x_i+i y_i$. Note that 
$dz_i\wedge d\overline{z}_i=-2i dx_i\wedge dy_i$, which implies $\dholoc{\Theta}{U_i}=
\frac{i}{2}\dholoc{h}{U_i}dz_i\wedge d\overline{z}_i=\dholoc{h}{U_i}dx_i\wedge dy_i=
\dholoc{\Phi}{U_i}$ (say). The family of $\{\dholoc{\Phi}{U_i}\}_{i\in I}$ satisfies compatibility 
condition similar to \ref{funda_form}. Hence, $\Phi\in \mathcal{A}^2(L)$, which is called volume form.

\begin{remark}\rm{
In general, one can define fundamental form associated to a given $d$-complex manifold with 
$d$-Hermitian metric. A $d$-complex manifold with $d$-Hermitian metric $h$ and associated 
fundamental form $\Theta$ is called $d$-K\"ahler manifold, if $d\Theta=0$. Note that Klein surfaces
are $d$-K\"ahler manifolds.
}
\end{remark}  

A given $d$-Hermitian metric on Klein surface $(X,\dholos{X})$ induces $d$-Hermitian metric on bundle 
$\mathcal{A}^{p,q}_d$ of $d$-smooth $(p,q)$ forms on $X$ for $0\leq p,q\leq 1$, which is described 
as follows:

Let $\phi,\psi\in \mathcal{A}^{p,q}_d$ and $(U_i,z_i)$ be chart on $X$. Let $\phi|_{U_j}=
\dholoc{\phi}{U_j}(dz_j)^p\wedge (d\overline{z}_j)^q$ and 
$\psi|_{U_j}=\dholoc{\psi}{U_j}(dz_j)^p\wedge (d\overline{z}_j)^q$, then 
$$
\dholoc{\langle \phi, \psi\rangle}{U_j}=
\dholoc{\phi}{U_j}\dholoc{\overline{\psi}}{U_j}\bigl\langle (dz_j)^p
\wedge (d\overline{z}_j)^q,(dz_j)^p\wedge (d\overline{z}_j)^q\bigl\rangle.
$$

Using the orientation line bundle, we have the Hodge star operator 
$\star \colon \mathcal{A}^r \ra \mathcal{A}^r(L)$ (see \cite[\S 2]{W}). 
Since $\mathcal{A}_d^r = \mathcal{A}^r \oplus i\mathcal{A}^{2-r}(L)$, 
one can extend the Hodge star operator to $\star \colon \mathcal{A}^r_d \ra \mathcal{A}_d^{2-r}$ 
as follows:
If $\omega = \sigma + i\tau$, then $\star \omega := \star \sigma +i \star \tau$, where 
$\star \sigma \in \mathcal{A}^{2-r}(L)$ and $\star \tau \in \mathcal{A}^{2-r}$.

For a $d$-complex valued $(p,q)$ form $\psi$ on Klein surface, we can have a $d$-complex valued 
$(1-q,1-p)$ form $\star\psi$ such that
$$
\langle\phi,\psi\rangle\Phi=\phi\wedge \star\overline{\psi}
$$
where $\star$ is Hodge star operator.

Let $\phi,\psi\in \mathcal{A}^{p,q}_d$ such that $\dholoc{\phi}{U_j}=
\dholoc{f}{U_j}(dz_j)^p\wedge (d\overline{z}_j)^q$ and $\dholoc{\psi}{U_j}=
\dholoc{g}{U_j}(dz_j)^p\wedge (d\overline{z}_j)^q$, where $0\leq p,q\leq 1$.
Then, $\dholoc{\langle\phi,\psi\rangle}{U_j}=\dholoc{f}{U_j}\dholoc{\overline{g}}{U_j}
\bigl\langle (dz_j)^p\wedge (d\overline{z}_j)^q,(dz_j)^p\wedge (d\overline{z}_j)^q\bigl\rangle$ and 
\[
\begin{array}{ll}
\dholoc{\langle\phi,\psi\rangle}{U_j}\dholoc{\Phi}{U_j}&=
\dholoc{f}{U_j}\dholoc{\overline{g}}{U_j}
\bigl\langle (dz_j)^p\wedge (d\overline{z}_j)^q,(dz_j)^p
\wedge (d\overline{z}_j)^q\bigl\rangle
\{\frac{i}{2}\dholoc{h}{U_j}dz_j\wedge d\overline{z}_j\}\\
\\
&=\dholoc{\phi}{U_j}\wedge (-1)^{(1-p)q}
\frac{i}{2}\dholoc{\overline{g}}{U_j}\dholoc{h}{U_j}
\bigl\langle (dz_j)^p\wedge (d\overline{z}_j)^q,(dz_j)^{(1-p)}\wedge (d\overline{z}_j)^{(1-q)}\bigl\rangle.
\end{array}
\]

This implies that
$$
\star\dholoc{\overline{\psi}}{U_j}=(-1)^{(1-p)q}\frac{i}{2}\dholoc{\overline{g}}{U_j}
\dholoc{h}{U_j}\bigl\langle (dz_j)^p\wedge (d\overline{z}_j)^q,(dz_j)^p\wedge (d\overline{z}_j)^q\bigl\rangle
(dz_j)^{(1-p)}\wedge (d\overline{z}_j)^{(1-q)}
$$ 
Also, we have $\dholoc{\overline{\psi}}{U_j}=\dholoc{\overline{g}}{U_j}(d\overline{z}_j)^p\wedge 
(dz_j)^q=(-1)^{pq}\dholoc{\overline{g}}{U_j}(dz_j)^q\wedge (d\overline{z}_j)^p
$

This further implies that
$$
\star\dholoc{\psi}{U_j}=(-1)^q\frac{i}{2}\dholoc{g}{U_j}\dholoc{\overline{h}}{U_j}
\bigl\langle (dz_j)^q\wedge (d\overline{z}_j)^p,(dz_j)^q\wedge (d\overline{z}_j)^p\bigl\rangle
(dz_j)^{(1-q)}
\wedge (d\overline{z}_j)^{(1-p)}
$$ 
Note  that $\phi\in \mathcal{A}^{p,q}_d$, so we have the compatibility condition as follows
\begin{equation}\label{pqformcomp}
    \dholoc{g}{U_j} = \left\{\begin{array}{lr}
        \dholoc{g}{U_i}(\frac{\partial z_i}{\partial z_j})^p
        (\frac{\partial \overline{z}_i}{\partial \overline{z}_j})^q,\hspace{10pt}  
        \text{if } z_j\circ z_i^{-1} \text{ is holomorphic}\\
        \\
        \dholoc{g}{U_i}(\frac{\partial z_i}{\partial \overline{z}_j})^q
        (\frac{\partial \overline{z}_i}{\partial z_j})^p, \hspace{10pt} 
        \text{if } z_j\circ z_i^{-1} \text{ is anti-holomorphic}
        \end{array}\right. 
\end{equation}
Hence, using \ref{hermicomp} and \ref{pqformcomp}, we have the compatibility condition as follows
\[
    \dholoc{\star\psi}{U_j} = \left\{\begin{array}{lr}
        \dholoc{\star\psi}{U_i},\hspace{10pt}  
        \text{if } z_j\circ z_i^{-1} \text{ is holomorphic}\\
        \\
        -\dholoc{\overline{\star\psi}}{U_i}, \hspace{10pt} 
        \text{if } z_j\circ z_i^{-1} \text{ is anti-holomorphic}
        \end{array}\right. 
\]
i.e., $\star\psi\in \mathcal{A}_d^{(1-q,1-p)}(L)$ for a given $d$-complex valued $(p,q)$ form 
$\psi\in \mathcal{A}_d^{(p,q)}$, where ${0\leq p,q\leq 1}$. It can be easily check that 
$(\star\overline{\psi})=\overline{(\star\psi)}$.

\section{Atiyah exact sequence}\label{sec-doperator}

\subsection{First order differential $d$-operators}
For given two $d$-holomorphic vector bundles $E$ and $F$ over a Klein surface $(X,\dholos{X})$, 
we will denote by $Hom_{\twistC}(E,F)=[\dholoa{U},P_{\dholoa{U}}]$ i.e. maximal family 
$\{\dholoc{P}{U_i}:E|_{U_i}\ra F|_{U_i}\}_{i\in I}$ of $\C$-linear morphisms. Let $[\dholoc{P}{U_i}]$ 
denotes a matrix associated to morphism $\dholoc{P}{U_i}$, then we have the following compatibility 
condition
\begin{equation}\label{section3equation1}
[\dholoc{P}{U_j}] = \left\{\begin{array}{lr}
    \cocyclebundle{j}{i}{F}[\dholoc{P}{U_i}]\cocyclebundle{i}{j}{E}=
    (\cocyclebundle{i}{j}{F})^{-1}[\dholoc{P}{U_i}]\cocyclebundle{i}{j}{E},  
    ~\text{if } z_j\circ z_i^{-1} \text{ is holomorphic}\\
       \\
    \cocyclebundle{j}{i}{F}\overline{[\dholoc{P}{U_i}]}\overlinecocyclebundle{i}{j}{E}=
    (\overlinecocyclebundle{i}{j}{F})^{-1}\overline{[\dholoc{P}{U_i}]}\overlinecocyclebundle{i}{j}{E},  
    ~\text{if } z_j\circ z_i^{-1} \text{ is anti-holomorphic}
    \end{array}\right. 
\end{equation}
Note that the above compatibility condition is imposed from the gluing condition \eqref{section2equation1}, which is
described as follows:
Let $(U_i,z_i)$ and $(U_j,z_j)$ be two charts on $X$ and $z_j\circ z_i^{-1}$ is anti-holomorphic. 
Let $s=[\dholoa{U},s_{\dholoa{U}}]\in \Gamma(X,E)$ and $P=[\dholoa{U},P_{\dholoa{U}}]\in Hom_{\twistC}(E,F)$. 
On $U_i\cap U_j$, $[\dholoc{P}{U_i}][\dholoc{s}{U_i}]=
[\dholoc{P}{U_i}]\cocyclebundle{i}{j}{E}\overline{[\dholoc{s}{U_j}]}$. Also 
$[\dholoc{P}{U_j}][\dholoc{s}{U_j}]=\cocyclebundle{j}{i}{F}\overline{[\dholoc{P}{U_i}]}\overline{[\dholoc{s}{U_i}]}
=\cocyclebundle{j}{i}{F}\overline{[\dholoc{P}{U_i}]}\overlinecocyclebundle{i}{j}{E}[\dholoc{s}{U_j}]$.
Similary, if $z_j\circ z_i^{-1}$ is holomorphic then we  will get, 
$[\dholoc{P}{U_j}]=\cocyclebundle{j}{i}{F}[\dholoc{P}{U_i}]\cocyclebundle{i}{j}{E}$.

The presheaf $\bigl(U\ra Hom_{\twistC}(E|_U,F|_U)\bigl)$ forms a sheaf, which will be denoted by $\mathcal{H}om_{\twistC}(E,F)$.

\begin{remark}\rm{
The above-described sheaf $\mathcal{H}om_{\twistC}(E,F)$ has both left as well as right 
$\dhstruct{X}$-module structure.
Let $[\dholoa{U}\cap U,f_{\dholoa{U}\cap U}]\in \dhstruct{X}(U)$ and 
$[\dholoa{U}\cap U,P_{\dholoa{U}\cap U}]\in \mathcal{H}om_{\twistC}(E,F)(U)$ for some open set $U\in X$, 
then $[\dholoa{U}\cap U,f_{\dholoa{U}\cap U}].[\dholoa{U}\cap U,P_{\dholoa{U}\cap U}]=
[\dholoa{U}\cap U,(f.P)_{\dholoa{U}\cap U}]$. Similarly, 
$[\dholoa{U}\cap U,P_{\dholoa{U}\cap U}].[\dholoa{U}\cap U,f_{\dholoa{U}\cap U}]=
[\dholoa{U}\cap U,(P.f)_{\dholoa{U}\cap U}]$.
}
\end{remark}
  
Using both $\dhstruct{X}$-module structure and for a given open subset $U\in X$, we can define the 
Lie bracket $[P,f]\in \mathcal{H}om_{\twistC}(E,F)(U)$ as follows
$$
[P,f]=Pf-fP
$$
for $P\in \mathcal{H}om_{\twistC}(E,F)(U)$ and $f\in \dhstruct{X}(U)$.

We will denote by $\mathcal{H}om_{\dhstruct{X}}(E,F)$, as sheaf of $\dhstruct{X}$-linear morphism 
between $d$-holomorphic bundles $E$ and $F$, which is also $\dhstruct{X}$-submodule of 
$\mathcal{H}om_{\twistC}(E,F)$. Also note that an element $P\in \mathcal{H}om_{\twistC}(E,F)(U)$ 
is $\dhstruct{X}(U)$-linear if $[P,f]=0$, where $U$ is an open subset of $X$ and $f\in \dhstruct{X}(U)$.

For a given open subset $U\subset X$ with  $V\subset U$ open, the Lie bracket commutes with restrictions 
i.e. $\restr{[P,f]}{V}=[\restr{P}{V},\restr{f}{V}]$, where $P\in \mathcal{H}om_{\twistC}(E,F)(U)$ and 
$f\in \dhstruct{X}(U)$.

\begin{definition}\label{s1d2}\rm{
For given two $d$-holomorphic vector bundles $E$ and $F$ over $(X,\dholos{X})$, a first-order 
differential $d$-operator is an element $P\in Hom_{\twistC}(E,F)$ such that the Lie bracket 
$[P|_U,f]:\restr{E}{U}\ra \restr{F}{U}$ is $\dhstruct{X}(U)$ linear for any open subset $U$ of $X$ and 
$f\in \dhstruct{X}(U)$.
}
\end{definition}

An element $P\in \mathcal{H}om_{\twistC}(E,F)$  is a first-order differential $d$-operator if and only if 
$[[P,f],g]\equiv 0$ for $f,g\in \dhstruct{X}$. Every $\dhstruct{X}(U)$-linear sheaf morphism 
$P:E|_U\ra F|_U$ satisfies $[P,f]\equiv 0$, for $f\in \dhstruct{X}(U)$, therefore $P$ is a 
first-order differential $d$-operator. Let $\mathcal{D}_1(\restr{E}{U},\restr{F}{U})$ be a subset of 
$Hom_{\twistC}(\restr{E}{U},\restr{F}{U})$ whose members are first-order differential $d$-operators. 
More generally, for every open subset $U$ of $X$ define,
$$
\mathcal{D}_1(E,F)(U)=\mathrm{Diff}_1(E|_U,f|_U)
$$
Then, $\mathcal{D}_1(E,F)(U)$ is an $\dhstruct{X}(U)$-submodule of $\mathcal{H}om_{\twistC}(E,F)(U)$.

If $V\subset U$, and if $P\in \mathcal{H}om_{\twistC}(E,F)(U)$ belongs to $\mathcal{D}_1(E,F)(U)$, 
then $P|_V$ will also belong to $\mathcal{D}_1(E,F)(V)$. 

Let $\{V_j\}_{j\in J}$ be a family of open set in $X$, and $U={\displaystyle\bigcup_jV_j}$. 
If $P\in \mathcal{H}om_{\twistC}(E,F)(U)$ such that $P|_{V_j}\in \mathcal{D}_1(E,F)(V_j)$ for all 
$j\in J$, then $P$ belongs to $\mathcal{D}_1(E,F)(U)$ which is true using the fact that Lie bracket 
commutes with restriction. Hence, $\mathcal{D}_1(E,F)$ is a subsheaf of $\mathcal{H}om_{\twistC}(E,F)$, 
which will be denoted by $\mathcal{D}_1(E,F)$ and $\mathcal{D}_1(E,F)(U)=\mathrm{Diff}_1(\restr{E}{U},\restr{F}{U})$.

The proof of the following result follows on the same line of arguments as in the proof of \cite[Proposition 2.6]{IBNR}.
\begin{proposition}\label{section3proposition10}
Let $P\in \mathcal{H}om_{\twistC}(E,F)$. Then, the following statements are equivalent:
\begin{enumerate}
\item $P$ is a first-order differential $d$-operator between $d$-holomorphic bundles $E$ and $F$ of ranks $r_1$ and $r_2$, respectively.
\item The first order differential $d$-operator is locally $\C$ linear. For every point $x\in X$, we have $P_x(\mathfrak{m}_x^2E_x)\subset \mathfrak{m}_xF_x$, where $\mathfrak{m}_x$ denotes the maximal ideal in $\dhstruct{X,x}$, $E_x$ denotes the stalk of sheaf associated to $E$ at $x$ and $P_x:E_x\ra F_x$ is the $\R$-linear map between stalks $E_x$ and $F_x$ induced by $P$.
\item \label{p3} Let $(U_i,z_i)$ be a chart around $x\in X$, let $s=\{s_1,s_2,\dots s_{r_1}\}$ and $t=\{t_1,t_2,\dots t_{r_2}\}$ be frames for $E|_{U_i}$ and $F|_{U_i}$, respectively. Then there exists $d$-holomorphic functions $(\dholoc{a}{U_i})^{i'}_l,(\dholoc{b}{U_i})^{i'}_l\in \dhstruct{X}(U_i)$ $(1\leq i'\leq r_1, 1\leq l\leq r_2)$ such that,
\begin{equation}\label{s1eq1}
\dholoc{P}{U_i}(\displaystyle\sum_{l=1}^{r_1} \dholoc{f}{U_i l}s_{l})=
\displaystyle\sum_{l=1}^{r_1}\sum_{i'=1}^{r_2} (\dholoc{a}{U_i})^{i'}_l\dholoc{f}{U_i l}t_{i'}+
\sum_{l=1}^{r_1}\sum_{i'=1}^{r_2} (\dholoc{b}{U_i})^{i'}_l\frac{\partial \dholoc{f}{U_i l}}{\partial {z_i}}t_{i'}
\end{equation}
where $\dholoc{f}{U_i l}\in \dhstruct{X}(U_j)$ $(1\leq l\leq r_1)$.
\item For every point $x\in X$, there exists a chart $(U_i,z_i)$, frames $s=\{s_1,s_2,\dots s_{r_1}\}$ and $t=\{t_1,t_2,\dots t_{r_2}\}$ of $E|_{U_i}$ and $F|_{U_i}$ respectively and $d$-holomorphic functions $(\dholoc{a}{U_i})^{i'}_l,(\dholoc{b}{U_i})^{i'}_l\in \dhstruct{X}(U_i)$ such that equation \eqref{s1eq1} holds for $\dholoc{f}{U_i l}\in \dhstruct{X}(U_i)$ $(1\leq l\leq {r_1})$.
\end{enumerate}
\end{proposition}

\subsection*{Symbol exact sequence}
We will now define the symbol map associated to a first order differential $d$-operator.

\begin{proposition}\label{section3proposition11}
Let $P:E\ra F$ be first order differential $d$-operator, then there exist a unique 
$\dhstruct{X}$-module homomorphism, $\sigma_1(P):\dhform{d}\ra \mathcal{H}om_{\dhstruct{X}}(E,F)$ 
such that for any open subset $V\subset X$,
\begin{equation}\label{section3equation3}
\sigma_1(P)_{V}(df)=[P_{V},f]:\restr{E}{V}\ra \restr{F}{V}\text{ for }f\in \dhstruct{X}(V)
\end{equation} 
\end{proposition}
\begin{proof} 
Note that $\dhform{d}|_{U_i}$ is generated by $dz_i$, where $(U_i,z_i)$ is some chart on $X$. 
Hence, equation \eqref{section3equation3} determines $\sigma_1(P)$ uniquely. Also, the local existence 
of $\sigma_1(P)$ on chart $(U_i,z_i)$, satisfying equation \eqref{section3equation3} can be shown 
following the same line of arguements as in \cite[Proposition 2.7]{IBNR} using Proposition 
\ref{section3proposition10}, which is given as
\begin{equation}\label{section3equation4}
\dholoc{\sigma_1(P)}{U_i}(\dholoc{df}{U_i})=
[\dholoc{P}{U_i},\dholoc{f}{U_i}]:\restr{E}{U_i}\ra 
\restr{F}{U_i}\text{ for }\dholoc{f}{U_i}\in \dhstruct{X}(U_i)
\end{equation}
Note that corresponding to a given $d$-holomorphic atlas $\dholoa{U}=\{(U_i,z_i)\}_{i\in I}$, 
we have a family $\{\dholoc{\sigma_1(P)}{U_i}(dz_i)\}_{i\in I}$ satisfying the following 
compatibility condition
$$
[\sigma_1(P)_{U_j}(dz_j)] = \left\{\begin{array}{lr}
  \frac{\partial z_j}{\partial z_i}(x)\cocyclebundle{j}{i}{F}[\sigma_1(P)_{U_i}(dz_i)]
  \cocyclebundle{i}{j}{E},\hspace{5pt}  ~\text{if } z_j\circ z_i^{-1} \text{ is holomorphic}\\
        \\
  \frac{\partial z_j}{\partial \overline{z}_i}(x)\cocyclebundle{j}{i}{F}\overline{[\sigma_1(P)_{U_i}(dz_i)]}  
  \overlinecocyclebundle{i}{j}{E}, \hspace{5pt} ~ \text{if } z_j\circ z_i^{-1} \text{ is anti-holomorphic}
  \end{array}\right. 
$$
or
\begin{equation}\label{section3equation5}
[\sigma_1(P)_{U_j}(dz_j)] = \left\{\begin{array}{lr}
  \frac{\partial z_j}{\partial z_i}(x)(\cocyclebundle{i}{j}{F})^{-1}[\sigma_1(P)_{U_i}(dz_i)]
  \cocyclebundle{i}{j}{E}, \hspace{5pt}  ~ \text{if } z_j\circ z_i^{-1} \text{ is holomorphic}\\
        \\
  \frac{\partial z_j}{\partial \overline{z}_i}(x)(\overlinecocyclebundle{i}{j}{F})^{-1}
  \overline{[\sigma_1(P)_{U_i}(dz_i)]}\overlinecocyclebundle{i}{j}{E}, \hspace{5pt} 
  ~\text{if } z_j\circ z_i^{-1} \text{ is anti-holomorphic}
        \end{array}\right. 
\end{equation}
where $[\dholoc{\sigma_1(P)}{U_j}(dz_j)]$ is used to represent matrix associated to 
$\dhstruct{X}$-linear morphism $\dholoc{\sigma_1(P)}{U_j}(dz_j)$ and the description of compatibility 
condition is below.

Let $(U_i,z_i)$ and $(U_j,z_j)$ be two charts on $X$ and $x\in U_i\cap U_j$. First we will consider 
that the transition map $z_j\circ z_i^{-1}$ is anti-holomorphic, let $V$ be a neighbourhood of $x$ 
such that $V\subset U_i\cap U_j$. Using Taylor's series expansion for transition function 
$z_j\circ z_i^{-1}$ on $V$, we have
\[
\begin{array}{ll}
z_j&=z_j(x)+\frac{\partial(z_j)}{\partial \overline{z}_i}(x)\bigl(\overline{z}_i-\overline{z}_i(x)\bigl)+\dots \\ 
\end{array}
\]
Hence,
\[
\begin{array}{ll}
\dholoc{P}{U_j}\bigl(z_j-z_j(x)\bigl)&=\frac{\partial(z_j)}{\partial \overline{z}_i}(x)\cocyclebundle{j}{i}{F}
\overline{[\dholoc{P}{U_i}]}\{\bigl(\overline{z}_i-\overline{z}_i(x)\bigl)+\dots\}\overlinecocyclebundle{i}{j}{E}\\ \\
&=\frac{\partial(z_j)}{\partial \overline{z}_i}(x)\cocyclebundle{j}{i}{F}\overline{[[\dholoc{P}{U_i},z_i]]}
\overlinecocyclebundle{i}{j}{E}\hspace{5pt}\text{ using Proposition }\ref{section3proposition10}\\ 
\end{array}
\]
i.e.,
\[
\begin{array}{ll}
[\dholoc{\sigma_1(P)}{U_j}(dz_j)]&=\frac{\partial(z_j)}{\partial \overline{z}_i}(x)\cocyclebundle{j}{i}{F}
\overline{[\dholoc{\sigma_1(P)}{U_i}(dz_i)]}\overlinecocyclebundle{i}{j}{E}\hspace{5pt}\text{ using equation }
\eqref{section3equation4}
\end{array}
\]
Similarly, if tranisition map $z_j\circ z_i^{-1}$ is holomorphic, then we have
$$
[\dholoc{\sigma_1(P)}{U_j}(dz_j)]=
\frac{\partial(z_j)}{\partial z_i}(x)\cocyclebundle{j}{i}{F}[\dholoc{\sigma_1(P)}{U_i}(dz_i)]
\cocyclebundle{i}{j}{E}
$$
where $[[\dholoc{P}{U_i},z_i]]$ is used to denote matrix associated to $\dhstruct{X}$-linear morphism 
$[\dholoc{P}{U_i},z_i]$ between $E|_{U_i}$ and $F|_{U_i}$.

Now, let $V$ be an arbitrary open subset of $X$ then we have a family 
$\{\dholoc{\sigma_1(P)}{U_i\cap V}(\dholoc{df}{U_i\cap V})\}_{i\in I}$ for a given 
$f=[\dholoa{U}\cap V,f_{\dholoa{U}\cap V}]\in \dhstruct{X}(V)$ with the following compatibility 
condition which can described using compatibility condition \eqref{section3equation5}
$$
[\dholoc{\sigma_1(P)}{U_j\cap V}(\dholoc{df}{U_j\cap V})] = \left\{\begin{array}{lr}
  \cocyclebundle{j}{i}{F}[\sigma_1(P)_{U_i\cap V}(\dholoc{df}{U_i\cap V})]
  \cocyclebundle{i}{j}{E},\hspace{5pt}  ~\text{if } z_j\circ z_i^{-1} \text{ is holomorphic}\\
        \\
  \cocyclebundle{j}{i}{F}\overline{[\sigma_1(P)_{U_i\cap V}(\dholoc{df}{U_j\cap V})]}
  \overlinecocyclebundle{i}{j}{E}, \hspace{5pt} ~\text{if } z_j\circ z_i^{-1} \text{ is anti-holomorphic}
  \end{array}\right. 
$$
or,
$$
[\sigma_1(P)_{U_j\cap V}(\dholoc{df}{U_j\cap V})] = \left\{\begin{array}{lr}
       (\cocyclebundle{i}{j}{F})^{-1}[\sigma_1(P)_{U_i\cap V}(\dholoc{df}{U_i\cap V})]
       \cocyclebundle{i}{j}{E}, \hspace{5pt}  \text{if } z_j\circ z_i^{-1} \text{ is holomorphic}\\
        \\
        (\overlinecocyclebundle{i}{j}{F})^{-1}\overline{[\sigma_1(P)_{U_i\cap V}(\dholoc{df}{U_i\cap V})]}
        \overlinecocyclebundle{i}{j}{E}, \hspace{5pt} \text{if } z_j\circ z_i^{-1} \text{ is anti-holomorphic}
        \end{array}\right. 
$$
Hence, the the equivalence class of family $\{\dholoc{\sigma_1(P)}{U_i\cap V}(\dholoc{df}{U_i\cap V})\}_{i\in I}$, 
which will be denoted by $\dholoc{\sigma_1(P)}{V}$ belongs to $\mathcal{H}om_{\twistC}(E,F)(V)$ and 
$$
\sigma_1(P)_V(df)=\{\dholoc{\sigma_1(P)}{U_i\cap V}(\dholoc{f}{U_i\cap V}\}_{i\in I}=
\{[\dholoc{P}{U_i\cap V},\dholoc{f}{U_i\cap V}]\}_{i\in I}=[P|_V,f]
$$
for $P\in \mathcal{D}_1(E,F)$ and $f\in \dhstruct{X}(V)$.
\end{proof}
  
\begin{definition}\label{s1d3}\rm{
For a given first order differential $d$-operator $P:E\ra F$, the unique $\dhstruct{X}$-module 
homomorphism $\sigma_1(P)$ defined in Proposition \ref{section3proposition11} along with the 
compatibility condition \eqref{section3equation5} is called the symbol of first order differential 
$d$-operator $P$.
}
\end{definition}

\begin{remark}\rm{
Let $E$ and $F$ be two $d$-holomorphic bundles or rank $r_1$ and $r_2$ respectively. Locally, 
on a chart $(U_i,z_i)$ on $X$, let bundle $E|_{U_i}$ and $F|_{U_i}$ have frames 
$\{s_{il}\}_{1\leq l\leq r_1}$ and $\{t_{il}\}_{1\leq l\leq r_2}$ then local description of 
symbol map can be given as follows. 

For any neighbourhood $V_i\subset U_i$ around some point $x\in X$, let 
$\dholoc{\omega}{U_i\cap V_i}\in \dhform{d}(V_i)$ then there exists 
$\dholoc{\xi}{U_i\cap V_i}:V_i\ra \C$, a holomorphic function such that 
$\dholoc{\omega}{U_i\cap V_i}=\dholoc{\xi}{U_i\cap V_i}dz_i$, and
\[
\begin{array}{ll}
\dholoc{\sigma_1(P)}{U_i}(\dholoc{\omega}{U_i\cap V_i})(s_{il})
&= \dholoc{\xi}{U_i\cap V_i}\dholoc{\sigma_1(P)}{U_i}(dz_i)(s_{il})\\ \\
&=\dholoc{\xi}{U_i\cap V_i}[\dholoc{P}{U_i}|_{V_i},z_i|_{V_i}](s_{il})\\ \\

& =\dholoc{\xi}{U_i\cap V_i}\displaystyle\sum_{m=1}^{r_2}t_{im}(\dholoc{b}{U_i})^m_l 
\end{array}\]
where $[\dholoc{P}{U_i}|_{V_i},z_i|_{V_i}]\in \mathcal{H}om_{\twistC}(E_{V_i},F_{V_i})$ and 
$((\dholoc{b}{U_i})^m_l)$ $(1\leq m\leq r_1, 1\leq l\leq r_2)$ is matrix associated to morphism 
$[\dholoc{P}{U_i}|_{V_i},z_i|_{V_i}]\in \mathcal{H}om_{\twistC}(E_{V_i},F_{V_i})$, with respect to frames 
$\{s_{il}\}_{1\leq l\leq r_1}$ and $\{t_{im}\}_{1\leq m\leq r_2}$ of $E_{U_i}$ and $F_{U_i}$, 
respectively.

In other words, each entry of the matrix of $\dhstruct{X}(U_i)$-module homomorphism 
$\dholoc{\sigma_1(P)}{U_i}$ with respect to $d$-holomorphic frames $\{s_i\}_{1\leq i\leq r_1}$ 
and $\{t_i\}_{1\leq i\leq r_2}$ for $E|_{U_i}$ and $F|_{U_i}$, respectively is holomorphic 
along with the compatibility conditions \eqref{section3equation4}. Hence symbol map is $d$-holomorphic. 
}
\end{remark}
For a given differential $d$-operator $P\in \mathrm{Diff}_1(E,F)$, symbol map $\sigma_1(P)$ can be considered 
as a global section of bundle $T_dX\otimes_{\dhstruct{X}}\mathcal{H}om_{\dhstruct{X}}(E,F)$ using 
the following sheaf isomorphism
$$
\mathcal{H}om_{\dhstruct{X}}\bigl(\dhform{d}, \mathcal{H}om(E,F)\bigl)\simeq 
T_dX\times \mathcal{H}om_{\dhstruct{X}}(E,F)
$$
In other words, we have a map
$$
\sigma_1:\mathcal{D}_1(E,F)\ra T_dX\otimes_{\dhstruct{X}}\mathcal{H}om_{\dhstruct{X}}(E,F)
$$
which is surjective. Hence, we get the following symbol exact sequence
\begin{equation}\label{section3equation6}
0\ra \mathcal{H}om_{\dhstruct{X}}(E,F)\xra{i}\mathcal{D}_1(E,F)\xra{\sigma_1} 
T_dX\otimes_{\dhstruct{X}}\mathcal{H}om_{\dhstruct{X}}(E,F)\ra 0.
\end{equation}

\subsection{$d$-holomorphic connections}
Let $E$ be a $d$-holomorphic vector bundle $E$ over a Klein surface $(X,\dholos{X})$. There is a 
canonical $\dhstruct{X}$-module monomorphism 
$i:T_dX\ra T_dX\otimes_{\dhstruct{X}}\mathcal{E}nd_{\dhstruct{X}}(E)$ given by 
$i(X)=X\otimes_{\dhstruct{X}}\textbf{1}_E$. Hence $T_dX(U)$ can be identified as an 
$\dhstruct{X}(U)$-submodule of $T_dX\otimes_{\dhstruct{X}}\mathcal{E}nd_{\dhstruct{X}}(E)(U)$. 
Also, we have the following decomposition 
\begin{equation}\label{section4equation1}
T_dX\otimes_{\dhstruct{X}}\mathcal{E}nd_{\dhstruct{X}}(E)=
\bigl(T_dX\oplus T_dX\otimes_{\dhstruct{X}}\mathcal{E}nd^0_{\dhstruct{X}}(E)\bigl)
\end{equation}
where $\mathcal{E}nd^0_{\dhstruct{X}}(E)$ is the kernel of the following trace map
$$
Tr:\mathcal{E}nd_{\dhstruct{X}}(E)\ra \dhstruct{X}
$$
Recall that the symbol map associated with a given first-order differential $d$-operator 
$P$, $\sigma_1(P)$ is a global section of $T_dX\otimes_{\dhstruct{X}}\mathcal{E}nd_{\dhstruct{X}}(E)$, 
but we are interested in those first-order differential $d$-operators whose symbol map is the global 
section of $T_dX$, considering $T_dX$ as an $\dhstruct{X}$-submodule of $T_dX\otimes_{\dhstruct{X}}\mathcal{E}nd_{\dhstruct{X}}(E)$ under the canonical monomorphism discussed above.

\begin{definition} \label{section4definition1}\rm{
For a given $d$-holomorphic vector bundle $E$ over Klein surface $(X,\dholos{X})$, a derivative 
endomorphism of $E$ is a first-order differential $d$-operator $P\in \mathcal{D}_1(E)$, whose 
symbol $\sigma_1(P)$ is a global section of $T_dX$.
}
\end{definition}

\begin{proposition}\label{section4proposition2}
A first-order differential $d$-operator $P\in \mathrm{Diff}_1(E)$ is a derivative endomorphism if and only if 
for every open subset $U$ of $X$, there exists a $d$-holomorphic vector field $Y\in T_dX(U)$ 
such that

\begin{equation}\label{DE}
P_U(fs)=fP_U(s)+(Yf)s
\end{equation}

for $f\in \dhstruct{X}(U)$ and $s\in E(U)$ and, the symbol morphism on $U$ is given by
\[\sigma_1(P)_U(\omega)=\langle Y,\omega\rangle 1_E\]
where $\omega\in \dhform{d}(U)$ and $\langle,\rangle:T_dX(U)\otimes \dhform{d}(U)\ra \dhstruct{X}(U)$ is canonical pairing.
\end{proposition}
\begin{proof}
Let $P$ be a derivative endomorphism and $\dholoa{U}=\{(U_i,z_i)\}_{i\in I}$ be dianalytic atlas on $X$ then by using the same line of arguments as in \cite[Proposition 3.2]{IBNR}, we have
$$
\dholoc{P}{U_i\cap U}(\dholoc{f}{U_i\cap U}\dholoc{s}{U_i\cap U})=\dholoc{f}{U_i\cap U}\dholoc{P}{U_i\cap U}(\dholoc{s}{U_i\cap U})+(\dholoc{Y}{U_i\cap U}\dholoc{f}{U_i\cap U})\otimes 1_E (\dholoc{s}{U_i\cap U})
$$ 
\[
\begin{aligned}
\text{where, } f&=[\dholoa{U}\cap U,\dholoc{f}{\dholoa{U}\cap U}]\in \dhstruct{X}(U)\\ P&=[\dholoa{U}\cap U,\dholoc{P}{\dholoa{U}\cap U}]\in \mathcal{E}nd_{\twistC}(E)\\
s&=[\dholoa{U}\cap U,\dholoc{s}{\dholoa{U}\cap U}]\in \Gamma(U,E),\\
\text{ and }Y&=[\dholoa{U}\cap U,\dholoc{Y}{\dholoa{U}\cap U}]\in \Gamma(U,T_dX)
\end{aligned}
\]
Note that, the following compatibility condition holds
$$
\dholoc{f}{U_j\cap U}[\dholoc{P}{U_j\cap U}] = \left\{\begin{array}{lr}
       \dholoc{f}{U_i\cap U}(\cocyclebundle{i}{j}{E})^{-1}[P_{U_i\cap U}]\cocyclebundle{i}{j}{E},\hspace{5pt}  \text{if } z_j\circ z_i^{-1} \text{ is holomorphic}\\
        \\
        \dholoc{\overline{f}}{U_i\cap U}(\overlinecocyclebundle{i}{j}{E})^{-1}\overline{[P_{U_i\cap U}]}\overlinecocyclebundle{i}{j}{E}, \hspace{5pt} \text{if } z_j\circ z_i^{-1} \text{ is anti-holomorphic}
        \end{array}\right. 
$$
and
$$
(\dholoc{Y}{U_j\cap U}\dholoc{f}{U_j\cap U})\otimes 1_{E_{U_j}} = \left\{\begin{array}{lr}
       \frac{\partial z_j}{\partial z_i}(\dholoc{Y}{U_i\cap U}\dholoc{f}{U_i\cap U})\otimes 1_{E_{U_i}},\hspace{5pt}  \text{if } z_j\circ z_i^{-1} \text{ is holomorphic}\\
        \\
        \frac{\partial z_j}{\partial \overline{z}_i}(\dholoc{\overline{Y}}{U_i\cap U}\dholoc{\overline{f}}{U_i\cap U})\otimes 1_{E_{U_i}}, \hspace{5pt} \text{if } z_j\circ z_i^{-1} \text{ is anti-holomorphic.}
        \end{array}\right. 
$$
Hence, $fP_U+(X,f)\in \mathcal{E}nd_{\twistC}(E)(U)$ and $P_U\equiv fP_U+(X,f)$. Also, we have $\dholoc{\sigma_1(P)}{U_j}(dz_j)=\langle Y_j,dz_j\rangle$ for each $j\in I$ with the following compatibility condition
\begin{equation}\label{ACC}
\langle Y_j,dz_j\rangle\otimes 1_{E_{U_j}} = \left\{\begin{array}{lr}
       \frac{\partial z_j}{\partial z_i}\langle Y_i,dz_i\rangle\otimes 1_{E_{U_i}},\hspace{5pt}  \text{if } z_j\circ z_i^{-1} \text{ is holomorphic}\\
        \\
        \frac{\partial z_j}{\partial \overline{z}_i}\langle \overline{Y}_i,d\overline{z}_i\rangle\otimes 1_{E_{U_i}}, \hspace{5pt} \text{if } z_j\circ z_i^{-1} \text{ is anti-holomorphic.}
        \end{array}\right. 
\end{equation}
Hence, $\dholoc{\sigma_1(P)}{U}(\xi)=\langle Y,\xi \rangle \otimes 1_{E_U}$ for all $\xi\in \dhform{d}(U)$.

Conversely, let $\dholoc{P}{U}\in \mathcal{D}_1(E)$ such that equation \eqref{DE} holds.
Then, we have $\dholoc{\sigma_1(P)}{U}(df)=\langle Y,f \rangle\otimes 1_{E_U}=(Y,f)\otimes 1_{E_U}$. On each chart $(U_i,z_i)$, we have $\dholoc{\sigma_1(P)}{U_i}(dz_i)=\langle Y_i,dz_i\rangle\otimes 1_{E_{U_i}}$ with the compatibility condition \eqref{ACC}. Hence, $\dholoc{\sigma_1(P)}{U}(\xi)=\langle Y,\xi\rangle\otimes 1_{E_U}$ for all $\xi\in \dhform{d}(U)$. 
The rest of the proof follows in the same line of arguments as in \cite[Proposition 3.2]{IBNR}.
\end{proof}

Let $\mathcal{A}t_d(E)$ denote the collection of all derivative endomorphisms for a $d$-holomorphic 
vector bundle $E$. The $\dhstruct{X}$-submodule $\mathcal{A}t_d(E)$ of $\mathcal{D}_1(E)$ is 
called the \emph{Atiyah algebra} of $E$. 

The symbol exact sequence \eqref{section3equation6} and direct sum decomposition \eqref{section4equation1} 
gives the following exact sequence,
\begin{equation}\label{section4equation2}
0\ra \mathcal{E}nd_{\dhstruct{X}}(E)\ra \mathcal{A}t_d(E)\ra T_dX\ra 0
\end{equation}
which we call the \emph{Atiyah exact sequence} for $d$-holomorphic vector bundle $E$.

\begin{definition}\label{section5definition1}\rm{
A $d$-holomorphic connection in $E$ is an $\dhstruct{X}$-module homomorphism 
$$
\nabla:T_dX\ra \mathcal{A}t_d(E)
$$ which is the splitting of Atiyah exact sequence \eqref{section4equation2}. 
}
\end{definition}
For a given $d$-holomorphic vector field $Y$ over some open subset $U$ of $X$, we denote $\nabla_U(Y)\in \mathcal{A}t_d(E)(U)$ by $(\nabla_Y)_U$.

Thus, $(\nabla_Y)_U:E|_U\ra E|_U$ is derivative endomorphism and for $s\in \Gamma(U,E)$, $f\in \dhstruct{X}(U)$ we have (from Definition \ref{section5definition1} and Proposition \ref{section4proposition2}),
\begin{itemize}
\item[1.]$(\nabla_{fY})_U(s)=f(\nabla_Y)_U(s)$
\item[2.]$(\nabla_Y)_U(fs)=(Yf)s+f(\nabla_Y)_U(s)$
\end{itemize}
\begin{proposition}
Let $\nabla$ be a $d$-holomorphic connection in $d$-holomorphic vector bundle $E$ over $X$. Then, there exists a unique $\R$ linear sheaf morphism $d_{\nabla}:\Omega_d^{0}(E)\ra \Omega_d^{1}(E)$ of degree 1 such that,
\begin{itemize}
\item[1.] If $\alpha\in \Omega^0_d(U)$ and $\phi\in \Omega^{0}_d(E)(U)$, $U$ open in $X$, then 
$$
d_{\nabla}(\alpha\wedge\phi)=(d\alpha)\wedge\phi+\alpha\wedge(d_{\nabla}\phi)
$$
\item[2.] For all $s\in E(U)$ and $Y\in T_dX(U)$, we have $(d_{\nabla}s)(Y)=\nabla_Ys$
\end{itemize}
\end{proposition}
\begin{proof}
The uniqueness of $\R$-linear sheaf morphism $d_{\nabla}$ can be proved using the same line of arguements as in \cite[Proposition 4.2]{IBNR}. So it is sufficient to prove existence.

\noindent \textbf{Existence:}
First, we need to prove local existence(on some chart $(U_i,z_i)$), which is easy following to same line of arguements as in \cite[Proposition 4.2]{IBNR} and given as follows

Let $\{s_{il}\}_{1\leq l\leq r}$ be frame for $E|_{U_i}$ and for $\phi=\displaystyle\sum_{l=1}^r \dholoc{\phi}{U_il} \otimes s_{il}\in \Omega^0_d(E)$, we have
$$
\begin{array}{ll}
\dholoc{d_{\nabla}}{U_i}(\phi)&=\displaystyle\sum_{m=1}^r  d(\dholoc{\phi}{U_im})\otimes s_{im}+(\displaystyle\sum_{m=1}^r \dholoc{\phi}{U_im}\displaystyle\sum_{l=1}^r(\dholoc{\omega}{U_i})^l_m s_{il})
\end{array}
$$
and 
$$
\dholoc{d_{\nabla}}{U_i}(s_{im})=\displaystyle\sum_ls_{il} (\dholoc{\omega}{U_i})^l_m 
$$
where $(\dholoc{\omega}{U_i})^l_m\in \dhform{d}(U_i)$ with the following compatibility condition
\begin{equation}\label{connectioncompatibility}
    \dholoc{\omega}{U_j} = \left\{\begin{array}{lr}
        (\cocyclebundle{i}{j}{E})^{-1}.\dholoc{\omega}{U_i}.\cocyclebundle{i}{j}{E}+(\cocyclebundle{i}{j}{E})^{-1}. d(\cocyclebundle{i}{j}{E}),\hspace{10pt}  \text{if } z_{j}\circ z_{i}^{-1} \text{ is holomorphic}\\
        \\
        (\overlinecocyclebundle{i}{j}{E})^{-1}.\dholoc{\overline{\omega}}{U_i}.\overlinecocyclebundle{i}{j}{E}+(\overlinecocyclebundle{i}{j}{E})^{-1}. d(\overlinecocyclebundle{i}{j}{E}), \hspace{10pt} \text{if } z_{j}\circ z_{i}^{-1} \text{ is anti-holomorphic}
        \end{array}\right. 
\end{equation}
which can be deduced as follows

Let two charts $(U_i,z_i)$, $(U_j,z_j)$ on $X$ and frames for $E_{U_i}$, $E_{U_j}$ be $\{s_{il}\}$, $\{s_{ij}\}$ $(1\leq l\leq r)$, respectively. Let $\cocyclebundle{j}{i}{E}:U_i\cap U_j\ra GL(r,\C)$ be co-cycle map satisfying the relation \eqref{cocycle}.

First, let us consider the case when $z_j\circ z_i^{-1}$ is anti-holomorphic. Then, we have
\[
\begin{array}{ll}
\dholoc{d_{\nabla}}{U_j}(s_{jm})&=\dholoc{d_\nabla}{U_i}\bigl(\displaystyle\sum_l 
\overline{s}_{il}(\overlinecocyclebundle{i}{j}{E})^l_m\bigl)\\
&=\displaystyle\sum_l d\bigl((\overlinecocyclebundle{i}{j}{E})^l_m\bigl)\overline{s}_{il}+
\displaystyle\sum_l (\overlinecocyclebundle{i}{j}{E})^l_m\dholoc{d_\nabla}{U_i}(\overline{s}_{il})\\
\end{array}
\]
Hence, 
\[
\begin{array}{ll}
\displaystyle\sum_l (\dholoc{\omega}{U_j})^l_m s_{jl}&=\displaystyle\sum_l 
d\bigl((\overlinecocyclebundle{i}{j}{E})^l_m\bigl)\displaystyle\sum_o s_{jo}(\cocyclebundle{j}{i}{E})^o_l+
\displaystyle\sum_l (\overlinecocyclebundle{i}{j}{E})^l_m
\displaystyle\sum_o\overline{s}_{io}(\dholoc{\overline{\omega}}{U_i})_l^o\\
&=\displaystyle\sum_o s_{jo} \Bigl(\cocyclebundle{j}{i}{E}d\bigl((\overlinecocyclebundle{i}{j}{E})^o_m\bigl)\Bigl)+
\displaystyle\sum_l (\cocyclebundle{j}{i}{E})^p_o\displaystyle\sum_o\sum_p s_{jp}
(\overlinecocyclebundle{i}{j}{E})^l_m(\dholoc{\overline{\omega}}{U_i})_l^o\\
\end{array}
\]
or
\[
\begin{array}{ll}
\displaystyle\sum_l s_{jl} {\dholoc{\omega}{U_j}}^l_m &=\displaystyle\sum_l s_{jl} 
\bigl(\cocyclebundle{j}{i}{E}d(\overlinecocyclebundle{i}{j}{E})\bigl)_{lm}+\displaystyle\sum_l s_{jl}
(\cocyclebundle{j}{i}{E}\dholoc{\overline{\omega}}{U_i} \overlinecocyclebundle{i}{j}{E})_{lm}\\
\end{array}
\]
Hence, 
$$
\dholoc{\omega}{U_j}=(\overlinecocyclebundle{i}{j}{E})^{-1}d(\overlinecocyclebundle{i}{j}{E})+
(\overlinecocyclebundle{i}{j}{E})^{-1}\dholoc{\overline{\omega}}{U_i}\overlinecocyclebundle{i}{j}{E}
$$
Similary, if $z_j\circ z_i^{-1}$ is holomorphic, then we have
$$
\dholoc{\omega}{U_j}=(\cocyclebundle{i}{j}{E})^{-1}d(\cocyclebundle{i}{j}{E})+
(\cocyclebundle{i}{j}{E})^{-1}\dholoc{\omega}{U_i}\cocyclebundle{i}{j}{E}
$$

Now let $s=[\dholoa{U}\cap V,s_{\dholoa{U}\cap V}]\in \Omega^0_d(E)(V)$ for some open subset $V$ 
of $X$ and $\dholoa{U}=\{(U_i,z_i\}_{i\in I}$ be atlas on $X$. If $\{s_{il}\}_{1\leq l\leq r}$ be 
frame for $E|_{U_i}$ then each $\dholoc{s}{U_i\cap V}$ can be written as 
$\displaystyle\sum_{l=1}^r{\dholoc{f}{U_i\cap V}}_ls_{il}$, where 
${\dholoc{f}{U_i\cap V}}_l\in \dhstruct{X}(U_i\cap V)$ and we have the family 
$\{\dholoc{d_{\nabla}}{U_i\cap V}(\dholoc{s}{U_i\cap V})\}_{i\in I}$ such that, if $z_j\circ z_i^{-1}$ 
is anti-holomorphic then we have,
\[
\begin{array}{ll}
\dholoc{d_{\nabla}}{U_j\cap V}(\displaystyle\sum_l{\dholoc{f}{U_j\cap}}_ls_{jl})&=
\displaystyle\sum_l s_{jl}d({\dholoc{f}{U_j\cap V}}_l)+
\displaystyle\sum_m\bigl(\sum_l s_{jl}(\dholoc{\omega}{U_j})^l_m\bigl){\dholoc{f}{U_j\cap V}}_m\\
&=\displaystyle\sum_l\bigl(\sum_m\overline{s}_{im}(\overlinecocyclebundle{i}{j}{E})^m_l\bigl)
d({\dholoc{f}{U_j\cap V}}_l)\\
&\quad \quad+\displaystyle\sum_n\Bigl(\sum_l\bigl(\sum_m\overline{s}_{im}(\cocyclebundle{i}{j}{E})^m_l\bigl)
[(\cocyclebundle{i}{j}{E})^{-1}\dholoc{\overline{\omega}}{U_i}
\overlinecocyclebundle{i}{j}{E}]^l_n\Bigl){\dholoc{f}{U_j\cap V}}_n\\
 &\qquad \qquad+\displaystyle\sum_n\Bigl(\sum_l\bigl(\sum_m\overline{s}_{im}
 (\overlinecocyclebundle{i}{j}{E})^m_l\bigl)[(\overlinecocyclebundle{i}{j}{E})^{-1}
 d(\overlinecocyclebundle{i}{j}{E})]^l_n\Bigl){\dholoc{f}{U_n\cap V}}_n\\
&=\displaystyle\sum_m\overline{s}_{im}\bigl(\displaystyle\sum_l
(\overlinecocyclebundle{i}{j}{E})^m_ld({\dholoc{\overline{f}}{U_j\cap V}}_l)\bigl)\\
&\quad\quad+\displaystyle\sum_m\overline{s}_{im}\bigl(\displaystyle\sum_l 
d(\overlinecocyclebundle{i}{j}{E})^m_l {\dholoc{f}{U_j\cap V}}_l\bigl)\\
&\qquad\qquad+\displaystyle\sum_m\overline{s}_{im}\bigl(\displaystyle\sum_l 
(\dholoc{\overline{\omega}}{U_i})^m_l(\overlinecocyclebundle{i}{j}{E}[\dholoc{S}{U_j\cap V}])_l\bigl)\\
&=\displaystyle\sum_m\overline{s}_{im} d(\overlinecocyclebundle{i}{j}{E}[\dholoc{S}{U_j\cap V}])_m+
\displaystyle\sum_m\overline{s}_{im}\bigl(\dholoc{\overline{\omega}}{U_i}(\overlinecocyclebundle{i}{j}{E}
[\dholoc{S}{U_j\cap V}])\bigl)_m\\
&=\displaystyle\sum \overline{s}_{im} d({\dholoc{\overline{f}}{U_i\cap V}}_m)+\displaystyle\sum_l
\bigl(\sum_m\overline{s}_{im}(\dholoc{\overline{\omega}}{U_i})^m_l\bigl){\dholoc{\overline{f}}{U_i\cap V}}_l=\overline{\dholoc{d_{\nabla}}{U_i\cap V}(\displaystyle\sum_l {\dholoc{f}{U_i\cap V}}_ls_{il})}
\end{array}
\]
and if $z_j\circ z_i^{-1}$ is holomorphic, then we have
$$
\dholoc{d_{\nabla}}{U_j\cap V}(\displaystyle\sum_l {\dholoc{f}{U_j\cap V}}_ls_{jl})=
\dholoc{d_{\nabla}}{U_i\cap V}(\displaystyle\sum_l {\dholoc{f}{U_i\cap V}}_ls_{il})
$$
Hence, we have the following compatibility condition
$$
    \dholoc{d_{\nabla}}{U_j\cap V}(\dholoc{s}{U_j\cap V})= \left\{\begin{array}{lr}
        \dholoc{d_{\nabla}}{U_i\cap V}(\dholoc{s}{U_i\cap V}),\hspace{10pt}  
        ~\text{if } z_{j}\circ z_{i}^{-1} \text{ is holomorphic}\\
        \\
        \overline{\dholoc{d_{\nabla}}{U_i\cap V}(\dholoc{s}{U_i\cap V})}, \hspace{10pt} 
        ~\text{if } z_{j}\circ z_{i}^{-1} \text{ is anti-holomorphic}
        \end{array}\right. 
$$
This shows that the family $\{\dholoc{d_{\nabla}}{U_i\cap V}(\dholoc{s}{U_i\cap V})\}_{i\in I}$ 
satisfies the condition \eqref{compatibilitycotangentsection}, i.e., we have 
$\dholoc{d_{\nabla}}{V}(s)=[\dholoa{U}\cap V,\dholoc{d_{\nabla}}{\dholoa{U}\cap V}
(s_{\dholoa{U}\cap V})]\in \dhform{d}(E)(V)$ and $\dholoc{d_{\nabla}}{V}(s)$ satisfies properties 
(1) and (2) because each $\dholoc{d_{\nabla}}{U_i\cap V}$ satisfies these properties.
\end{proof}

\subsection{\v{C}ech description of extension class}\label{s5ssec1}
Let 
$$
\alpha:0\ra E'\xrightarrow{i} E\xrightarrow{P} E'' \ra 0
$$
be an extension of $d$-holomorphic bundles $E''$ by $E'$.
There is an element $$\epsilon(\alpha)\in H^1(X,\mathcal{H}om(E'',E'),$$ which is unique for the 
equivalence class of isomorphic extensions, also known as the extension class. 
Let $\dholoa{U}=\{U_{\alpha},z_{\alpha}\}_{\alpha\in I}$ be a dianalytic atlas on $X$.

We have a short exact sequence
$$
0\ra \mathcal{H}om(E''|_{U_{\alpha}},E'|_{U_{\alpha}})\xrightarrow{i^*} 
\mathcal{H}om(E''|_{U_{\alpha}},E|_{U_{\alpha}})\xrightarrow{P^*}
\mathcal{E}nd(E''|_{U_{\alpha}})\ra 0.
$$
Let $1_{E''|_{U_{\alpha}}}\in \Gamma\bigl(X,\mathcal{E}nd(E''|_{U_{\alpha}})\bigl)$. Since $P^*$ is 
surjective, we have $h_{\alpha}:E''|_{U_{\alpha}}\ra E|_{U_{\alpha}}$ such that 
$P_{\alpha}\circ h_{\alpha}={1}_{E''|_{U_\alpha}}$ or $P^*(h_{\alpha})=1_{E''_{U_\alpha}}$.
Define 
\[
    \tilde{u}_{\alpha\beta} = \left\{\begin{array}{lr}
        {h_{\beta}}|_{(U_{\beta}\cap U_{\alpha})}-
{h_{\alpha}}|_{(U_{\beta}\cap U_{\alpha})},  \hspace{10pt} 
\text{if } z_{\beta}\circ z_{\alpha}^{-1} \text{ is holomorphic}\\
        \\
        {h_{\beta}}|_{(U_{\beta}\cap U_{\alpha})}-
{\overline{h}_{\alpha}}|_{(U_{\beta}\cap U_{\alpha})},  \hspace{10pt} 
\text{if } z_{\beta}\circ z_{\alpha}^{-1} \text{ is anti-holomorphic}
        \end{array}\right.
  \]
Then, 
\[
    P^*(\tilde{u}_{\alpha\beta}) = \left\{\begin{array}{lr}
        1_{E''_{U_{\beta}}}-1_{E''_{U_{\alpha}}},  \hspace{10pt} 
        \text{if } z_{\beta}\circ z_{\alpha}^{-1} \text{ is holomorphic}\\
        \\
        1_{E''_{U_{\beta}}}-\overline{1}_{E''_{U_{\alpha}}}, \hspace{10pt} 
        \text{if } z_{\beta}\circ z_{\alpha}^{-1} \text{ is anti-holomorphic}
        \end{array}\right.
  \]
or, $P^*(\tilde{u}_{\alpha\beta})\equiv 0$.

Because of the exactness of
$$
0\ra H^0\bigl(U_{\alpha}\cap U_{\beta},\mathcal{H}om(E'',E')\bigl)\ra 
H^0\bigl(U_{\alpha}\cap U_{\beta},\mathcal{H}om(E'',E)\bigl)
\xrightarrow{P^*}H^0\bigl(U_{\alpha}\cap U_{\beta},\mathcal{E}nd(E'')\bigl)
$$
we have a unique section $u_{\alpha\beta}\in \Gamma\bigl(U_{\alpha}\cap U_{\beta},\mathcal{H}om(E'',E')\bigl)$ 
such that $i^*(u_{\alpha\beta})=\tilde{u}_{\alpha\beta}$,
and on chart $(U_{\alpha}\cap U_{\beta}\cap U_{\gamma},z_{\gamma})$, we have 
\[
    i^*(u_{\alpha\gamma}) = \left\{\begin{array}{lr}
        i^*(u_{\alpha\beta})+i^*(u_{\beta\gamma}), \hspace{10pt} 
         \text{if } z_{\gamma}\circ z_{\beta}^{-1} \text{ is holomorphic}\\
        \\
        i^*(\overline{u}_{\alpha\beta})+i^*(u_{\beta\gamma}),  \hspace{10pt} 
        \text{if } z_{\gamma}\circ z_{\beta}^{-1} \text{ is anti-holomorphic}
        \end{array}\right.
  \]
or,
\[
    u_{\alpha\gamma} = \left\{\begin{array}{lr}
        u_{\alpha\beta}+u_{\beta\gamma},  \hspace{10pt} 
        \text{if } z_{\gamma}\circ z_{\beta}^{-1} \text{ is holomorphic}\\
        \\
        \overline{u}_{\alpha\beta}+u_{\beta\gamma},  \hspace{10pt} 
        \text{if } z_{\gamma}\circ z_{\beta}^{-1} \text{ is anti-holomorphic}
        \end{array}\right.
  \]
Therefore, the family of sections 
$\{u_{\alpha\beta}\}_{\alpha,\beta\in I}$ satisfy the 1-cocycle condition i.e., 
$\{u_{\alpha\beta}\}_{\alpha,\beta\in I}$ is a 1-cocycle of \v{C}ech complex 
$C^{\bullet}\bigl(\dholoa{U},\mathcal{E}nd(E'',E')\bigl)$. The extension class $\epsilon(a)$, which is an element of $H^1\bigl(X,\mathcal{E}nd(E'',E')\bigl)$, equals to \v{C}ech cohomology class of the 1-cocycle 
$\{u_{\alpha\beta}\}$.

The extension class of Atiyah exact sequence is called \emph{Atiyah class}. For a given $d$-holomorphic 
bundle $E$, the Atiyah class will be denoted by $at_d(E)$. 

From definition of the Atiyah class, it is obvious to note that $at_d(E)$ is an element of 
$H^1\Bigl(X,\mathcal{H}om\bigl(T_dX,\mathcal{E}nd_{\dhstruct{X}}(E)\bigl)\Bigl)$ but there is canonical isomorphism,
$$H^1\Bigl(X,\mathcal{H}om\bigl(T_dX,\mathcal{E}nd_{\dhstruct{X}}(E)\bigl)\Bigl) \cong 
H^1\Bigl(X,\dhform{d}\bigl(\mathcal{E}nd_{\dhstruct{X}}(E)\bigl)\Bigl)
$$
Further, using Dolbeault isomorphism \eqref{s4dolbeaultiso} we can consider 
$at_d(E)\in H^{(1,1)}\bigl(X,\mathcal{E}nd_{\dhstruct{X}}(E)\bigl)$. 

It is important to note that an extension splits if and only if the extension class vanishes, so a
$d$-holomorphic bundle $E$ admits $d$-holomorphic connection if $at_d(E)$ vanishes.

\subsection*{\v{C}ech description of Atiyah class}
Let $\dholoa{U}=\{U_i,z_i\}_{i\in I}$ be a $d$-holomorphic atlas of $X$. Note that $d$-holomorphic 
vector bundle $E$ (rank r) on $X$ is trivial bundle on $U_i$ for each $i\in I$, let 
$\{s_{i l}\}_{1\leq l\leq r}$ frame for $E|_{U_i}$. Note that, trivial bundle $E|_{U_i}$ always 
admits canonical connection defined as $
\Delta_i(\displaystyle\sum_{m=1}^r \dholoc{f}{U_i m}s_{i m})=\displaystyle\sum_{m=1}^r
\frac{\partial \dholoc{f}{U_{im}}}{\partial z_i} s_{im}\otimes dz_i$. Here each $\Delta_i$ is splitting of 
Atiyah exact sequence for $E|_{U_i}$. Take
\[
    \theta_{ij} = \left\{\begin{array}{lr}
        \Delta_j|_{(U_i\cap U_j)}-\Delta_i|_{(U_i\cap U_j)}, \hspace{10pt} 
         \text{if } z_{j}\circ z_{i}^{-1} \text{ is holomorphic}\\
        \\
        \Delta_j|_{(U_i\cap U_j)}-\overline{\Delta}_i|_{(U_i\cap U_j)},  \hspace{10pt} 
        \text{if } z_{j}\circ z_{i}^{-1} \text{ is anti-holomorphic}
        \end{array}\right.
  \] 

For two charts $(U_i,z_i)$ and $(U_j,z_j)$ on $X$, let $\cocyclebundle{j}{i}{E}: U_{i}\cap U_{j}\ra GL(r,\C)$ 
be co-cycle map for bundle $E$ with respect to 
frames $s_i=\{s_{i1},\dots s_{ir}\}$ and $s_j=\{s_{j1},\dots s_{jr}\}$ for $E_{U_i}$ and $E_{U_j}$, 
respectively. 

Using \eqref{cocycle1}, it can be easily check that the Atiyah class $at_d(E)$ of $d$-holomorphic 
bundle $E$ is given by 1 co-cycle 
$\theta_{ij}\in \Gamma\Bigl(U_i\cap U_j,\dhform{d}\bigl(\mathcal{E}nd_{\dhstruct{X}}(E)\bigl)\Bigl)$ satisfying the 
following conditions (cf. \cite[Section 4.4]{IBNR})
\begin{equation}\label{sec2Atiyah class}
    \theta_{ij}= \left\{\begin{array}{lr}
        -\bigl((\cocyclebundle{i}{j}{E})^{-1}d(\cocyclebundle{i}{j}{E})\bigl),\hspace{10pt} 
        ~ \text{if } z_{j}\circ z_{i}^{-1} \text{ is holomorphic}\\
        \\
       -\bigl((\overlinecocyclebundle{i}{j}{E})^{-1}d(\overlinecocyclebundle{i}{j}{E})\bigl), \hspace{10pt} 
       ~ \text{if } z_{j}\circ z_{i}^{-1} \text{ is anti-holomorphic}
        \end{array}\right. 
\end{equation}

\begin{proposition}\label{s4p7}
Let $E$ be a $d$-holomorphic bundle over Klein surface $X$ and $D$ be a $d$-smooth connection on 
$E$ compatible with $d$-holomorphic structure of $E$. Let $R$ denote the curvature of $D$. Then, 
the dolbeault isomorphism $\delta:H^{1,1}\bigl(X,\mathcal{E}nd_{\dhstruct{X}}(E)\bigl)\ra 
H^1\Bigl(X,\dhform{d}\bigl(\mathcal{E}nd_{\dhstruct{X}}(E)\bigl)\Bigl)$ carries the class $[R^{(1,1)}]$ to 
$-at_d(E)$. In particular $[R^{(1,1)}]$ is independent of connection $D$.
\end{proposition}
\begin{proof} The proof follows in the same line of arguements as in \cite[Proposition 4.5]{IBNR} 
using \eqref{cocycle1}, \eqref{s4eq1}, \eqref{connectioncompatibility} and the description of 
Atiyah class \eqref{sec2Atiyah class}.
\end{proof}

\begin{remark}\rm{
One may extend the concepts and results of this section to $d$-complex manifolds of higher dimensions. 
Presumably, the Hodge decomposition theorem remains valid for compact $d$-K\"ahler manifolds. Using it, 
one can show that if a $d$-holomorphic bundle $E$ admits a holomorphic connection, then all the 
Chern classes of $E$ vanish (cf. \cite[Theorem 3.6]{IBNR}). 
}
\end{remark}

\section{Jet bundles}\label{sec-Jet}
Now we will relate differential $d$-operators and holomorphic connections to jet bundles 
associated to $d$-holomorphic bundles.

Let $E$ be a $d$-holomorphic bundle on a Klein surface $X$. There is a natural $d$-holomorphic bundle
$J^1(E)$ such that the fibre $J^1(E)(x) = \bigslant{E_x}{\m{x}^2E_x}$, where $\m{x}$ is the maximal 
ideal in $\dhstruct{X,x}$ (cf. \cite[Chapter IV]{RSP}). 
Let $\dholoa{U}=\{U_i,z_i\}_{i\in I}$ be $d$-holomorphic atlas on $X$ and $\{U_i,\phi_i\}_{i\in I}$ 
be associated $d$-holomorphic trivialization of $E$. Using the arguements as in \cite{RSP}, we have 
$J^1(E)|_{U_i}\overset{\psi_i}\simeq U_i\times\C^r\oplus L(\C,\C^r)$, where $L(\C,\C^r)$ is the
space of linear maps. Let $s=[\dholoa{U}\cap U,\dholoc{s}{\dholoa{U}\cap U}]\in \Gamma(U,E)$ and 
$[\dholoa{U}\cap U,\dholoc{j_1(s)}{\dholoa{U}\cap U}]\in \Gamma\bigl(U,J^1(E)\bigl)$, then 
$\psi_i(\dholoc{j_1(s)}{U_i\cap U})(x)=\bigl(x,\dholoc{\xi}{U_i\cap U}(x)\bigl)$, where 
$\dholoc{\xi}{U_i\cap U}(x)=\bigl(\phi_i\circ \dholoc{s}{U_i}(x),\frac{\partial(\phi_i\circ 
\dholoc{s}{U_i})}{\partial z_i}(x)\bigl)\in \C^r\oplus L^1_s(\C,\C^r)$ with the compatibility 
condition as follows
\[
    \bigl(\phi_j\circ \dholoc{s}{U_j}(x),\frac{\partial(\phi_j\circ \dholoc{s}{U_j})}{\partial z_j}(x)\bigl) = \left\{\begin{array}{lr}
        \bigg(\cocyclebundle{j}{i}{E}\Big(z_i(x)\Big)\Big(\phi_i\circ \dholoc{s}{U_i}(x)\Big),\cocyclebundle{j}{i}{E}\Big(z_i(x)\Big)\Big(\frac{\partial(\phi_i\circ \dholoc{s}{U_i})}{\partial z_i}(x)\frac{\partial z_i}{\partial z_j}(x)\Big)\bigg),\\
\qquad\qquad\text{if } z_{j}\circ z_{i}^{-1} \text{ is holomorphic}\\
        \\
        \bigg(\cocyclebundle{j}{i}{E}\Big(z_i(x)\Big)\Big(\overline{\phi_i\circ \dholoc{s}{U_i}(x)}\Big),\cocyclebundle{j}{i}{E}\Big(z_i(x)\Big)\Big(\frac{\partial(\overline{\phi_i\circ \dholoc{s}{U_i}})}{\partial \overline{z}_i}(x)\frac{\partial \overline{z}_i}{\partial z_j}(x)\Big)\bigg),\\
        \qquad\qquad\text{if } z_{j}\circ z_{i}^{-1} \text{ is anti-holomorphic}
        \end{array}\right. 
  \]

For each open set $U \subset X$, there is a canonical map $\dholoc{(j_1)}{U}:E(U)\ra J^1(E)(U)$ 
such that for $s=[\dholoa{U}\cap U,\dholoc{s}{\dholoa{U}\cap U}]\in E(U)$, $\dholoc{(j_1)}{U}(s)(x)=
s_x+\mathfrak{m}_x^2E_x$ and $\dholoc{j_1(s)}{U}=\dholoc{(j_1)}{U}(s)=
[\dholoa{U}\cap U,\dholoc{j_1(s)}{\dholoa{U}\cap U}]$ with the compatibility condition as follows
\[
   \dholoc{j_1(s)}{U_j\cap U}(x) = \left\{\begin{array}{lr}
        \dholoc{j_1(s)}{U_i\cap U}(x),\hspace{10pt}  
        \text{if } z_j\circ z_i^{-1} \text{ is holomorphic}\\
        \\
         \dholoc{\overline{j_1(s)}}{U_i\cap U}(x), \hspace{10pt} 
         \text{if } z_j\circ z_i^{-1} \text{ is anti-holomorphic.}
        \end{array}\right. 
  \]
Given a section 
$[(\dholoa{U}\cap U,s_{\dholoa{U}\cap U})]\in \Gamma(U,E)$, we get a family of maps 
$j_{1,U_i\cap U}:U\cap U_{i}\ra J^1(E)$ with the following twisting conditions:
\[
    j_{1,U_j\cap U} = \left\{\begin{array}{lr}
        j_{1,U_i\cap U},\hspace{10pt}  
        ~\text{if } z_j\circ z_i^{-1} \text{ is holomorphic}\\
        \\
        \overline{j}_{1,U_i\cap U}, \hspace{10pt} 
        ~\text{if } z_j\circ z_i^{-1} \text{ is anti-holomorphic}
        \end{array}\right. 
  \]
i.e., we have a $d$-holomorphic section $[(\dholoa{U}\cap U,j_{1,\dholoa{U}\cap U})]\in 
\Gamma\bigl(U,J^1(E)\bigl)$.

The bundle $J^1(E)$ is called the $d$-holomorphic bundle of $1$-jets of $E$. 
If $(U_i,z_i)$ is a $d$-holomorphic chart on $X$ and $s_i=\{s_{i 1},s_{i2},\dots s_{il}\}$ be frame 
for $E_{U_i}$, then the family, 
$$
\{j_1(s_{il})|1\leq l\leq r\}\cup\{j_1(z_i s_{i l})|1\leq l\leq r\}
$$ 
is frame for $J^1(E)_{U_i}$. In particular, the rank of $J^1(E)$ is $2r$. The sheaf 
morphism $j_1:E\ra J^1(E)$ is $\R$-linear; however it is not $\dhstruct{X}$-linear.

Using Proposition \ref{section3proposition10} (\ref{p3}), it is easy to verify that we have

\begin{proposition}\cite[cf. Proposition 5.1]{IBNR}
The canonical morphism $j_1:E\ra J^1(E)$ is a first-order differential $d$-operator moreover for 
every $d$-holomorphic vector bundle $F$ over $X$, the map 
$j_1^*:\mathcal{H}om_{\dhstruct{X}}(J^1(E),F)\ra \mathcal{D}_1(E,F)$ is an isomorphism of 
$\dhstruct{X}$-modules.
\end{proposition}

The above proposition shows that the sheaf morphism $j_1:E\ra J^1(E)$ is a universal first-order 
differential $d$-operator, i.e. for a given $d$-holomorphic vector bundle $F$ over $X$ and any 
first-order differential $d$-operator $P\in \mathrm{Diff}_1(E,F)$, there exists a unique 
$\dhstruct{X}$-linear sheaf morphism $T_P:J^1(E)\ra F$ such that $T_P\circ j_1\equiv P$.

There is an $\R$-linear map $p(x):J^1(E)(x)\ra E(x)$ such that $p(x)\bigl(j_1(s)(x)\bigl)=s(x)$ 
which defines an $\dhstruct{X}$-linear sheaf morphism $p:J^1(E)\ra E$ such that for some open 
subset $U\subset X$, $\dholoc{p}{U}=[\dholoa{U}\cap U,\dholoc{p}{\dholoa{U}\cap U}]$ with the 
compatibility condition as follows
\[
   \dholoc{p}{U_j\cap U} = \left\{\begin{array}{lr}
        (\cocyclebundle{i}{j}{E})^{-1}\circ \dholoc{P}{U_i\cap U}\circ 
        \cocyclebundle{i}{j}{J^1(E)},\hspace{10pt}  
        \text{if } z_j\circ z_i^{-1} \text{ is holomorphic}\\
        \\
         (\overlinecocyclebundle{i}{j}{E})^{-1}\circ 
         \dholoc{\overline{P}}{U_i\cap U}\circ \overlinecocyclebundle{i}{j}{J^1(E)}, \hspace{10pt} 
         \text{if } z_j\circ z_i^{-1} \text{ is anti-holomorphic.}
        \end{array}\right. 
  \]
where,
\[
   \cocyclebundle{i}{j}{J^1(E)} = \left\{\begin{array}{lr}
        \begin{bmatrix}
\cocyclebundle{i}{j}{E} &0 \\
0 &\frac{\partial z_i}{\partial z_j}\cocyclebundle{i}{j}{E} 
\end{bmatrix},\hspace{10pt}  
        \text{if } z_j\circ z_i^{-1} \text{ is holomorphic}\\
        \\
         \begin{bmatrix}
\cocyclebundle{i}{j}{E} &0 \\
0 &\frac{\partial \overline{z}_i}{\partial z_j}\cocyclebundle{i}{j}{E} 
\end{bmatrix}, \hspace{10pt} 
         \text{if } z_j\circ z_i^{-1} \text{ is anti-holomorphic.}
        \end{array}\right. 
  \]

Also, we have a canonical $\dhstruct{X}$-linear sheaf morphism $i:\dhform{E}\ra J^1(E)$ such that 
for some open subset $U\subset X$, $\dholoc{i}{U}=[\dholoa{U}\cap U,\dholoc{i}{\dholoa{U}\cap U}]$ 
and $\dholoc{i}{U_i\cap U}(dz_i\otimes s_{il})(x)=j_1(z_is_{il})(x)-z_i(x)j_1(s_{il})(x)$, where 
$\{s_{il}\}$ $(1\leq l\leq r)$ is frame for $\dholoc{E}{U_i}$ and $(U_i,z_i)$ is chart on $X$.

Hence we have,
\begin{proposition}\cite[cf. Proposition 4.3]{IBNR}\label{sec3prop1}
The sequence,
\begin{equation}\label{s6eq1}
0\ra \dhform{E}\xrightarrow{i}J^1(E)\xrightarrow{p}E\ra 0
\end{equation}
is an exact sequence.
\end{proposition}

The extension class of the exact sequence \eqref{s6eq1} is an element of 
$H^1(X,E^*\otimes \dhform{E})\cong H^1\bigl(X,\dhform{d}(\mathcal{E}nd_{\dhstruct{X}}E)\bigl)$. Note that 
$at_d(E)$ is also an element of $H^1\bigl(X,\dhform{d}(\mathcal{E}nd_{\dhstruct{X}}E)\bigl)$.
\begin{proposition}
The extension class of the exact sequence \eqref{s6eq1} equals to $-at(E)$ in 
$H^1\bigl(X,\dhform{d}(\mathcal{E}nd_{\dhstruct{X}}E)\bigl)$.
\end{proposition}
\begin{proof} The proof follows in the same line of arguements as in \cite[Proposition 5.6]{IBNR} 
using \eqref{cocycle1}, \S \ref{s5ssec1} and description of the Atiyah class \eqref{sec2Atiyah class}.
\end{proof}

By definition, a $d$-holomorphic bundle $E$ admits a $d$-holomorphic connection if and only if the 
Atiyah exact sequence splits $d$-holomorphically. Hence, we have

\begin{proposition}\label{prop-Jet-split}
A $d$-holomorphic bundle $E$ admits $d$-holomorphic connection if and only if the exact sequence 
\eqref{s6eq1} splits $d$-holomorphically.
\end{proposition}

\section{The Atiyah-Weil criterion}\label{sec-criterion}
We are now in position to give a criterion of the existence of $d$-holomorphic connection in a 
$d$-holomorphic bundle $E$ over a compact Klein surface. Let $(X, \dholos{X})$ be a compact Klein surface.

\begin{lemma}\label{s6l2}
Let $[\dholoa{U},f_{\dholoa{U}}]$ be $d$-holomorphic function on a compact Klein surface $X$ then each 
$\dholoc{f}{U_{\alpha}}$ is constant real valued function for all $\alpha \in I$, where 
$\dholoa{U}=\{(U_{\alpha},z_{\alpha})\}_{\alpha\in I}$ is $d$-holomorphic atlas on $X$.
\end{lemma}
\begin{proof}
Let $p:\tilde{X}\ra X$ be double cover of $X$, where $\tilde{X}$ is a compact Riemann surface 
with an anti-holomorphic involution $\sigma$ such that the Klein surface $\tilde{X}/\sigma$
is isomorphic to $X$ (see \cite[Theorem 1.6.1]{NLANG} for the existence of such a cover). 
By \cite[Theorem 4.1]{W1}, we get a holomorphic function $\tilde{f} \colon \tilde{X}\ra \C$
such that $\tilde{f}\bigl(\sigma(\tilde{x})\bigl) = \overline{\tilde{f}(\tilde{x})}$ for all $\tilde{x}\in \tilde{X}$,
which is locally given by the $d$-holomorphic function $[\dholoa{U},f_{\dholoa{U}}]$. Since $\tilde{X}$
is compact, it follows that $\tilde{f}$ is constant, and hence we can conclude that each 
$\dholoc{f}{U_{\alpha}}$ is constant real valued function for all $\alpha \in I$.
\end{proof}

\begin{proposition}\label{s6prop3}
Let $E$ be an indecomposable vector bundle over compact Klein surface $X$. Then every global endomorphism $\phi=[\dholoa{U},\phi_{\dholoa{U}}]\in End_{\dhstruct{X}}(E)$ is of the form $[\dholoa{U},\lambda_{\dholoa{U}}I+\psi_{\dholoa{U}}]$, where $\dholoc{\lambda}{U_i}\in \C$ and $\dholoc{\psi}{U_i}\in End_{\dhstruct{X}}(E|_{U_i})$ with the following compatibility condition:
\[
   \dholoc{\lambda}{U_{\beta}}\id{E|_{U_{\beta}}} +\dholoc{\psi}{U_{\beta}} = \left\{\begin{array}{lr}
       \dholoc{\lambda}{U_{\alpha}}\id{E|_{U_{\alpha}}} +\dholoc{\psi}{U_{\alpha}},\hspace{10pt}  \text{if } z_{\beta}\circ z_{\alpha}^{-1} \text{ is holomorphic}\\
        \\
        \dholoc{\lambda}{U_{\alpha}}\id{E|_{U_{\alpha}}} +\dholoc{\psi}{U_{\alpha}}, \hspace{10pt} \text{if } z_{\beta}\circ z_{\alpha}^{-1} \text{ is anti-holomorphic}
        \end{array}\right. 
  \]
\end{proposition}
\begin{proof}
For any point $x\in U_{\alpha}$, let $P_{\alpha}(x,t)=\displaystyle\sum_{i=0}^r 
\dholoc{a}{U_{\alpha}i}(x)t^i\in \C[t]$ be characteristic polynomial of $\C$-linear map 
$\phi_{\alpha}:E|_{U_{\alpha}}\ra E|_{U_{\alpha}}$, where $r$ is rank of $E$, 
$\dholoc{a}{U_{\alpha}i}$ $(0\leq i\leq r)$ are $d$-holomorphic function on $U_{\alpha}$. 
By Lemma \ref{s6l2}, it follows that each $\dholoc{a}{U_{\alpha}i}$ will be 
constant real valued function. So each $P_{\alpha}(x,t)=P_{\alpha}(t)$ will be free from $x$. 

Let $\dholoc{\lambda}{U_{\alpha}}$ be the root of $P_{\alpha}(t)$ in $\C$ and consider the endomorphism $\dholoc{\psi}{U_{\alpha}}=\dholoc{\phi}{U_{\alpha}}-\dholoc{\lambda}{U_{\alpha}}\id{E|_{U_{\alpha}}}$ with the following twisting condition:
\[
    \dholoc{\psi}{U_{\beta}} = \left\{\begin{array}{lr}
        \dholoc{\psi}{U_{\alpha}},\hspace{10pt}  \text{if } z_{\beta}\circ z_{\alpha}^{-1} \text{ is holomorphic}\\
        \\
        \dholoc{\overline{\psi}}{U_{\alpha}}, \hspace{10pt} \text{if } z_{\beta}\circ z_{\alpha}^{-1} \text{ is anti-holomorphic}
        \end{array}\right. 
  \]
hence we have $[(\dholoa{U},\psi_{\dholoa{U}})]\in \mathrm{End}_{\dhstruct{X}}(E)$. Now, using Fitting decomposition \cite[Proposition 5.1]{IBNR}, we have, $E=\mathrm{Ker}(\psi^n)\oplus E^n(\psi)$. Since $E$ is indecomposable, so we must have either $\mathrm{Ker}(\psi^n)\equiv 0$ or $E^n(\psi)\equiv 0$. From this, we can conclude that $\psi$ must be nilpotent (cf. \cite[Proposition 5.3]{IBNR}).
\end{proof}

\begin{remark}\label{s6remak}\rm{
We say that a $d$-holomorphic bundle $E$ over a Klein surface $X$ is indecomposable if whenever 
$E_1$ and $E_2$ are $d$-holomorphic subbundle of $E$ such that $E=E_1\oplus E_2$ then either 
$E_1=0$ or $E_2=0$.
Note that $H^0\bigl(X,\mathcal{E}nd_{\dhstruct{X}}(E)\bigl)$
has finite dimension for a $d$-holomorphic bundle $E$ over a compact Klein surface \cite{W1}, 
hence Krull-Schmidt theorem holds.
}
\end{remark}

\subsection{Atiyah-Weil criterion}\label{subsec-criterion}
Let $E$ be a $d$-holomorphic bundle over Klein surface $(X,\mathfrak{X})$, and $E^*$ be its dual 
vector bundle. Let $\langle, \rangle\colon E\times E^*\ra \mathcal{O}^{dh}_X$ be canonical pairing.

The above pairing induces an $\mathcal{O}^{dh}_X$-bilinear sheaf morphism
$$
\mathcal{A}^p_d(E)\otimes \mathcal{A}^q_d(E^*)\ra \mathcal{A}^{p+q}_d.
$$

This $\mathcal{O}_X^{dh}$-bilinear morphism induces $\mathbb{R}$-bilinear map
\begin{equation}\label{map3}
S:H^1\bigl(X,\Omega^1_d(E)\bigl)\times H^0(X,E^*)\ra \mathbb{C}
\end{equation}
as follows.

Let $a\in H^1\bigl(X,\Omega_d^1(E)\bigl)$ and $u\in H^0(X,E^*)$. Using Dolbeault isomorphism, we can 
identify $H^1\bigl(X,\Omega^1_d(E)\bigl)$ with $H^{1,1}(X,E)$ and can consider $a$ as an element of 
$H^{1,1}(X,E)$. Let $a=[\alpha]$, where $\alpha\in \mathcal{A}^{1,1}_d(E)$ and 
$\overline{\partial}_E(\alpha)=0$. Also, note that $\alpha\wedge u\in \mathcal{A}^{1,1}_d\subset \mathcal{A}^2_d$.
So, we can write $\alpha\wedge u=\sigma_1+i\sigma_2$, where $\sigma_1\in \mathcal{A}^2$ 
and $\sigma_2=\mathcal{A}^2(L)$.
Define 
$$
S(a,u)=\int (\star\sigma_1)\Phi + i\int \sigma_2\;,
$$
where $\Phi\in A^2(L)$ is the volume form on Klein surface $X$.

The trace pairing
$$
\mathcal{E}nd_{\dhstruct{X}}(E)\times \mathcal{E}nd_{\dhstruct{X}}(E)\ra \dhstruct{X}
$$
is a non-degenerate $\R$-bilinear map, which induces an isomorphism between $\mathcal{E}nd_{\dhstruct{X}}(E)$ 
and $[\mathcal{E}nd_{\dhstruct{X}}(E)]^*$. Using this, we obtain a non-degenerate $\R$-bilinear pairing
$$
\tilde{S}:H^1\Bigl(X,\dhform{d}\bigl(\mathcal{E}nd_{\dhstruct{X}}(E)\bigl)\Bigl)\times H^0\bigl(X,\mathcal{E}nd_{\dhstruct{X}}(E)\bigl)\ra \C
$$
given by
\begin{equation}\label{sec6eq1}
\tilde{S}(a,u)=\int_X(\star\sigma_1)\Phi + i\int_X\sigma_2
\end{equation}
where $\sigma_1+i\sigma_2=tr(\alpha u)$, $\alpha\in A^{1,1}\bigl(\mathcal{E}nd_{\dhstruct{X}}(E)\bigl)\subset
\mathcal{A}^2\bigl(\mathcal{E}nd_{\dhstruct{X}}(E)\bigl)$ is Dolbeault representation of $a$ and 
$\sigma_1\in \mathcal{A}^2$, $\sigma_2\in \mathcal{A}^2(L)$.

\begin{proposition}\label{s6prop9}
Let $E$ be a $d$-holomorphic bundle over a compact Klein surface $X$, and let,
$$
\tilde{S}:H^{1,1}\bigl(X,\mathcal{E}nd_{\dhstruct{X}}(E)\bigl)\times H^0\bigl(X,\mathcal{E}nd_{\dhstruct{X}}(E)\bigl)\ra \C
$$
be trace pairing as defined in \ref{sec6eq1}, which is non-degenerate $\R$-bilinear pairing. 
Let $\phi\in H^0(X,\mathcal{E}nd_{\dhstruct{X}}(E)$, then,
\[
    \tilde{S}(at_d(E),\phi) = \left\{\begin{array}{lr}
        2\pi \sqrt{-1}\text{ deg}(E),\hspace{10pt}  \text{if } \phi = id_E\\
        \\
        0, \hspace{10pt} \text{if } \phi \text{ is nilpotent}
        \end{array}\right. 
  \]
where $at_d(E)$ is atiyah class of $E$ and $\text{deg}(E)=\int_X c_1(E)$, $c_1(E)\in H^2(X,L)$ is 
first chern class of $E$.
\end{proposition}
\begin{proof}
Let $h$ be fixed smooth $d$-Hermitian metric on $E$ and $D$ be unique $d$-smooth unitary connection 
compatible with $d$-holomorphic structure on $E$. Let $R$ be the curvature of the connection $D$, 
which is $d$-smooth $(1,1)$ form on $X$ taking values in $\mathcal{E}nd_{\dhstruct{X}}(E)$. Also 
note that $R$ represents Dolbeault cohomology class $[R]\in H^{1,1}\bigl(X,\mathcal{E}nd_{\dhstruct{X}}(E)\bigl)$. 
From \ref{s4dolbeaultiso} and Proposition \ref{s4p7} Dolbeault isomorphism between 
$H^1(X,\dhform{d}\bigl(\mathcal{E}nd_{\dhstruct{X}}(E)\bigl)$ and $H^{1,1}\bigl(X,\mathcal{E}nd_{\dhstruct{X}}(E)\bigl)$ 
carries $at_d(E)$ to $-[R]$.

Therefore, $\tilde{S}(at_d(E),\phi)=\int_X tr(-R\phi)$. Now, if $\phi=id_E$, then 
$\tilde{S}(at_d(E),\phi)=\int_X tr(-R)=2\pi i\int_X \frac{iR}{2\pi}=
2\pi i\int_X c_1(E)=2\pi i\text{ deg}(E)$.

Now, if $\phi$ is nilpotent, then it is easy to prove that $R\circ \phi$ is nilpotent on each chart 
(see \cite[Proposition 18]{MFA} or \cite[Proposition 5.10]{IBNR}), hence $tr(-R\phi)=0$ i.e.,
$\tilde{S}(at_d(E),\phi)=0$.
\end{proof}

\begin{theorem}\label{s7theorem10}
Let $E$ be an indecomposable $d$-holomorphic vector bundle over a compact Klein surface $X$. Then, 
$E$ admits $d$-holomorphic connection if and only if $\text{deg}(E)=0$.
\end{theorem}
\begin{proof}
By Proposition \ref{s6prop3}, every endomorphism $\phi\in H^0\bigl(X,\mathcal{E}nd_{\dhstruct{X}}(E)\bigl)$ 
can be expressed as $\phi_{\alpha}=\lambda_{\alpha}I+\psi_{\alpha}$ on each chart 
$(U_{\alpha},z_{\alpha})$ with appropriate twisting condition, where $\psi_{\alpha}$ is nilpotent 
for each $\alpha\in A$. Hence, $tr(\phi_{\alpha})=\lambda_{\alpha}rk(E)$. 

So, we have a global $d$-holomorphic function $[\dholoa{U},\lambda_{\dholoa{U}}]$ such that 
$\lambda_{\alpha}=\frac{tr(\phi_{\alpha}{\color{red}{)}}}{rk(E)}$ satisfying the twisting condition.

From Lemma \ref{s6l2} each $\lambda_{\alpha}$ will be constant real valued function. Hence, from 
Proposition \ref{s6prop9},
$$
\tilde{S}(at_d(E),\phi)=2\pi i\frac{tr(\phi)}{rk(E)}\text{ deg}(E)
$$
for all $\phi\in H^0\bigl(X,\mathcal{E}nd_{\dhstruct{X}}(E)\bigl)$ or, $\tilde{S}(at_d(E),\phi)=0$ for 
all $\phi$ iff $\text{deg}(E)=0$ but $\tilde{S}$ is non-degenerate which implies $at_d(E)=0$ 
if and only if $\text{deg}(E)=0$.
\end{proof}

\begin{theorem}
For a given $d$-holomorphic bundle $E$, let $E=E_1\oplus E_2\oplus\dots \oplus E_k$ be Remak decomposition over a compact Klein surface. Then $E$ admits $d$-holomorphic connection if and only if $\text{deg}(E_i)=0$ for all $i\in \{1,2\dots,k\}$.
\end{theorem}
\begin{proof}
Since the Jet bundle functor is additive, we have an analogue of \cite[Proposition 5.4]{IBNR} for $d$-holomorphic bundle
$E$. Hence, using Proposition \ref{prop-Jet-split} and Theorem \ref{s7theorem10}, the result follows.
\end{proof}


\end{document}